\documentclass{article}
\usepackage{amssymb}
\usepackage{amsmath}
\usepackage{amsthm}
\usepackage{epsfig}
\usepackage{mathrsfs}
\usepackage{graphicx}
\usepackage{color}

\newtheorem{theorem}{Theorem}[section]
\newtheorem{lemma}[theorem]{Lemma}
\newtheorem{proposition}[theorem]{Proposition}
\newtheorem{definition}[theorem]{Definition}

\newtheorem{rk&ex}[theorem]{Remarks \& Examples}

\def\NN{{\mathbb N}}

\def\RR{{\mathbb R}}

\def\E{\mathscr E}
\def\F{\mathscr F}
\def\J{\mathscr J}

\def\eps{\varepsilon}

\def\e{\varepsilon}
\def\vphi{\varphi}

\def\na{\nabla}
\def\pa{\partial}
\def\ds{\displaystyle}

\title{Liquid Drops sliding down an inclined plane}
\author{Inwon Kim\thanks{Department of Mathematics, UCLA, Los Angeles CA 90095, USA. E-mail: ikim@math.ucla.edu. Partially supported by NSF Grant DMS-0970072.} and Antoine Mellet \thanks{Department of Mathematics, University of Maryland,
College Park, MD 20742, USA. E-mail: mellet@math.umd.edu.  Partially supported by NSF Grant DMS-0901340}}

\begin{document}

\maketitle
\begin{abstract}
We investigate a one-dimensional model describing the motion of liquid drops sliding down an inclined plane (the so-called quasi-static approximation model). 
We prove existence and uniqueness  of a solution and investigate its long time behavior for both homogeneous and inhomogeneous medium (i.e. constant and non-constant contact angle). We also obtain some homogenization results.
\end{abstract}

\section{Introduction}
\subsection{Equilibrium drops and quasi-static approximation}
Consider a liquid drop lying on an inclined plane.  We introduce a coordinate system such that this plane is the $(x,y)$-plane, and we assume that the free surface of the drop (the liquid/vapor interface) can be described as the graph of a function $(x,y)\mapsto u(x,y)$.
For small droplets, the energy of an equilibrium drop can be approximated by
\begin{equation}\label{eq:JJE}
\J(u) = \sigma \int_{\RR^2}  \frac{1}{2}|\na u|^2 +\beta \chi_{\{u>0\}} \,dx\,dy +  \int_{\RR^2} \!\int_0^{u(x,y)} \Gamma  \, dz\, dx\, dy
\end{equation}
where  $\sigma$ denotes the surface tension coefficient and $\beta(x,y)$ is the relative adhesion coefficient between the fluid and the solid.
If the support plane  is inclined at an angle $\alpha$ to the horizontal in the $x$-direction, the gravitational potential $\Gamma$ can be written as
\begin{equation}\label{eq:gamma}
\Gamma = \rho g(z  \cos \alpha - x  \sin \alpha)
\end{equation}
and the Euler-Lagrange equation for the minimization of $\J$ with volume constraint is the following equation (known as the  Young-Laplace equation):
\begin{equation} \label{eq:euler}
\left\{
\begin{array}{ll}
-\Delta u = \lambda - u  \kappa  \cos\alpha + x  \kappa   \sin \alpha \quad & \mbox{ in } \{u>0\}\\[5pt]
\frac{1}{2} |\na u | ^2 = \beta & \mbox{ on } \pa\{u>0\}.
\end{array}
\right.
\end{equation}
with $\kappa= \rho g/\sigma$ and where $\lambda$ is a Lagrange multiplier associated with the volume constraint.
Note that the Dirichlet integral $\int |\na u|^2\,dx\,dy$ is an approximation  of 
the surface tension energy, which classically involves the perimeter $\int \sqrt{1+|\na u|^2}\, dx\,dy$  (see \cite{F}). 
This perimeter functional leads to a mean curvature operator which would be much more delicate to deal with than the Laplace operator appearing in \eqref{eq:euler}.

\vspace{10pt}

It is well-known (see \cite{F}) that when $\beta$ is constant and $\kappa>0$, then (\ref{eq:euler})  has no solution (this can be seen by multiplying the first equation in (\ref{eq:euler}) by $u_x$ and integrating over $\{u>0\}$ - it is also obvious that any translation of $u$ down the $x$-axis will decrease the energy $\J$). 

This means that on a perfectly homogeneous surface, any drop should  slide down the inclined plane, no matter how small the inclination of the plane $\alpha$. 
This is not the case when $\beta$ is a function of $x$. In that case, one expects some drops to stick to the inclined plane, at least for small inclination (or small volume). 
This sticking phenomenon was rigorously investigated by Caffarelli-Mellet \cite{CM} when $\beta$ is assumed to be periodic with period $\eps\ll1$.

In any case, large drops will always eventually slide down the inclined plane and it is the purpose of this paper to investigate the motion of such drops, when $\beta$ is constant (homogeneous surface) and when it is not (heterogeneous surface).
There is a  number of papers (see for instance \cite{Blake,G,HM,LeGrand} and references therein) in the fluid dynamics literature discussing the behavior of sliding drops. Interesting phenomena are observed. In particular, the shape of the drop can change drastically when the velocity is increased, and a singularity (corner) develops at the rear of the drop when the velocity exceed a certain critical value (and a cusp may form
at even higher velocity).

In this paper, we aim at studying such phenomena in the framework of a (relatively) simple model for the motion of liquid drops usually referred to as the quasi-static approximation.
We will see that this model does allow for singularity formation of the type discussed above.

\subsection{The quasi-static approximation model.}
In the quasi-static approximation regime,  the speed of the contact line is much slower than the capillary relaxation time, so that at each instant there is a balance between gravitational force and surface tension (see \cite{G,HM}).
The free surface of the drop is thus at equilibrium at all time, and can be described by a function $u(x,y,t)$ solution of: 
\begin{equation} \label{eq:motion3d-1}
-\Delta u = \lambda(t) - u \kappa \cos\alpha + x \kappa \sin \alpha \quad \mbox{ in } \{u(t)>0\},
\end{equation}
where the constant $\lambda(t)$ is determined by the volume constraint
$$ \int u(x,t)\, dx = V_0.$$

Along the free boundary (or contact line) $\pa\{u>0\}$, the equilibrium contact angle condition is assumed to be satisfied only at some microscopic scale, and the deviation of the ``apparent" contact angle $\theta=|\na u|$ from the equilibrium value $\theta_e=\sqrt{2\beta}$ is responsible for the motion of the contact line.
The relation between the velocity of the contact line and the contact angle $\theta$ is not a settled issue, and various velocity laws have been proposed (see \cite{G}, \cite{HM}, \cite{LeGrand}) and studied (see \cite{GlK}).
Following Blake and Ruschak  \cite{Blake}, we assume that the speed is proportional to $\theta^2-\theta_e^2$.
More precisely, we assume that the normal velocity $V$ of the contact line $\pa\{u>0\}$ is given by
\begin{equation} \label{eq:motion3d-1b}
V= \frac{1}{2} |\na u | ^2 - \beta \qquad \mbox{ on } \pa\{u>0\}.
\end{equation}
This particular velocity law has the advantage of leading to a gradient flow formulation for \eqref{eq:motion3d-1}-\eqref{eq:motion3d-1b}. This gradient flow formulation will be discussed in the next section, but let us already stress out that \eqref{eq:motion3d-1b} leads to an obvious problem when the drop is sliding down: In the rear of the drop (where $V<0$), the speed $|V|$ is bounded by $\beta$ while in the front of the drop (where $V>0$), the speed can increase without bounds (and is expected to increase if the volume of the drop, or the inclination increase).
This would lead to ever-increasing support of the drop, which is  incompatible with \eqref{eq:motion3d-1}  (indeed, we will see \eqref{eq:motion3d-1} does not admits positive solutions on large domains).
This problem can be fixed by noticing that \eqref{eq:motion3d-1b} should not be expected to hold in the degenerate case where $\na u=0$ along the free boundary.
In the next section, we will see that the gradient flow formulation naturally leads to an obstacle problem type formulation of the velocity law of the form (see \eqref{eq:motion3d-2})
\begin{equation}\label{eq:vel0}
 \min\left\{-V+\frac{1}{2} |\na u | ^2 - \beta\, ,\,  |\na u|\right\}=0 \quad \mbox{ on } \pa\{u>0\}.
\end{equation}

The free boundary problem  (\ref{eq:motion3d-1})-(\ref{eq:vel0}) has been studied by Grunewald and Kim \cite{GK} in the particular case $\alpha=0$. In that case, we always have $|\na u|>0$ along the free boundary, and so \eqref{eq:motion3d-1b} holds (instead of (\ref{eq:vel0})).
In the present paper, we are interested in the case $\alpha>0$, for which the degeneracy of $|\na u|$ cannot be ruled out.
In  that case, the existence of solutions (which is the main focus in \cite{GK}) is not the only interesting issue.
Indeed, one would also like to know whether a given drop will stick  or will slide down the plane. And when it does slide down the plane, one would like to characterize the motion of the drop. 
We will see that there exist some traveling wave type solutions which describe the asymptotic  speed and profile of any sliding drops. Finally, one would like to determine the effects of heterogeneities of the inclined plane on the motion of the sliding drops. We refer to  \cite{K,K2} for results concerning   the motion of liquid droplets on a {\it horizontal} heterogenous plane (see also \cite{KM1} for similar results in the framework of Hele-Shaw flow).

Such issues are rather delicate to investigate in full generality. In this paper, we develop the analysis of  (\ref{eq:motion3d-1})-(\ref{eq:vel0}) (and answer all the above questions) in the one-dimensional case. 
In particular, we prove the existence and uniqueness of solutions (for general functions $\beta$) and we study their long time behavior when $\beta$ is constant, and when $\beta$ is periodic.
This one-dimensional model could be interpreted as  modeling the motion of a fixed volume of fluid initially placed uniformly across the inclined plane (ridge of fluid). However, such a configuration is not usually stable (the leading edge may become unstable in the span-wise direction as it flows down the inclined plane). We thus prefer to think of this study as a starting point for a more general study involving the higher dimensional case.

As we will see, even this (simpler) case already leads to interesting behaviors. In particular, the profile of the drop may degenerate in the rear of the drop (touching the ground tangentially).
Of course,  in dimension two and higher,  the dynamics of the drops is much more complex since   topology changes  seem unavoidable (splitting and merging of droplets).

Finally, let us mention that one of main difficulty in studying this problem (regardless of the dimension)
is the lack of comparison  principle due to the volume constraint (no viscosity solutions), since $\lambda$ depends on $t$ in \eqref{eq:motion3d-1}. 
And there is no way to get rid of this volume constraint, especially when addressing the long-time behavior of the drops. Indeed if we were to fix $\lambda(t)=\lambda_0$ and solve \eqref{eq:motion3d-1}-(\ref{eq:vel0}), it is easy to check that in most cases the drop would simply  vanish in finite time.

\vspace{10pt}

\subsection{Gradient flow structure} 
As mentioned above, the system of equation \eqref{eq:motion3d-1}-\eqref{eq:motion3d-1b} has a natural gradient flow structure:
For a given function $v$, we denote by
$$ \F(v) = \int_{\RR^2} \frac{1}{2} |\na v|^2 + \frac{v^2}{2} \kappa \cos\alpha -v \, x\kappa \sin\alpha\, dx\, dy $$
the energy of the corresponding drop, and for a given Caccioppoli set $W\subset \RR^2$, we  define
\begin{equation}\label{eq:uom} 
u_W = \mbox{argmin}\{\F(v)\, ;\, v\in H^1_0(W), \;  v\geq 0 ,\; \int v = V_0 \}.
\end{equation}
When $\alpha=0$, it is easy to show that the solution of (\ref{eq:motion3d-1}) with support $W$ and volume $V_0$ is nonnegative in  $W$. 
When $\alpha>0$, this may not be the case for large set $W$. This is the reason why we need the additional obstacle condition $v\geq 0$ in \eqref{eq:uom}.
As a consequence, we may not have  $\mbox{supp }u_W = W$, 
which is the main difference with the case $\alpha=0$.

We now define the energy of a  Caccioppoli set $W$ as:
$$ \E(W) = \F(u_W) + \int_W \beta \, dx\, dy$$
Formally, (\ref{eq:motion3d-1})-(\ref{eq:motion3d-1b}) is the gradient flow for the energy $\E$  over the set of Caccioppoli sets with respect to the usual Riemannian structure.
Indeed, let us consider a solution $W(t)$ of the gradient flow
$$ \pa_t W = -\mbox{diff } \E(W)$$
and denote by  $\widetilde V$ the normal velocity of $W$.
We can show that $\widetilde V$ satisfies  (\ref{eq:motion3d-1b})  as long as $W=\mbox{supp} u_W$:
Let $\tilde v$ be a  normal velocity field  applied to $\partial W$ and let  $\delta \tilde{u}$ be the variation of $u_W$ induced by $\tilde{v}$. Assuming that 
$|Du_W|\neq 0 \hbox{ on } \pa W$,
we get
$$
\begin{array}{ll}
\ds diff \E(W)(\tilde{v})& =\ds  
\int_W Du_W\cdot D\delta\tilde{u}+  (u_W \kappa \cos\alpha -  x\kappa \sin\alpha)\delta\tilde{u} \, dx +  \int_{\partial W}\left[\beta + \frac{1}{2}|D\tilde{u}_W|^2\right] \tilde{v} \, dS\\ [10pt]
&= \ds \int_W \big[-\Delta  u_W +u_W \kappa \cos\alpha -  x\kappa \sin\alpha\big]\delta\tilde{u}\, dx + \int_{\partial W} \left[-|Du_W|\delta\tilde{u}+ \left(\beta +\frac{1}{2} |D\tilde{u}_W|^2\right)\tilde{v} \right]\, dS\\ [10pt]
&= \ds \lambda \int_W \delta\tilde{u}\,  dx + \int_{\partial W}  \left[-|Du_W|^2+ \left(\beta +\frac{1}{2} |D\tilde{u}_W|^2\right) \right] \tilde{v} \,dS \\[10pt]
&=\ds  \int_{\partial W} \left(\beta-\frac{1}{2}|Du_W|^2\right) \tilde{v}\, dS,
 \end{array}
$$
where in the first equality we used the fact that $u_W=0$ on $\partial W$,  and the third inequality is a consequence of the following expansion:
\begin{equation}\label{assumption}
\delta\tilde{u}  = \tilde{v}|Du_W| +o(\tilde{v})\hbox{ on }\partial W.
\end{equation}
This implies
$$ \widetilde V=\frac{1}{2}|Du_W|^2-\beta.$$ 
However, as mentioned above, when $\alpha>0$, the solution of (\ref{eq:uom}) may have a support strictly smaller than $W$ when $W$ is large (because gravity is pushing $u$ "downward", in the direction of increasing $x$).
In that case, the obstacle free boundary condition yields $|\na u_W |=0$ 
along  $\pa\{u_W>0\} \setminus \pa W$ and we 
do not recover  (\ref{eq:motion3d-1b}).  
If we denote by $V$ the normal velocity of $\mbox{supp }u_W$, it is clear that whenever  $\mbox{supp }u_W \neq W$, we have $V \leq \widetilde V$.
Since this can only happen when $|\na u_W|=0$, we obtain the following equation  for the velocity $V$:
 \begin{equation} \label{eq:motion3d-2}
 \min\left\{-V+\frac{1}{2} |\na u_W | ^2 - \beta\, ,\,  |\na u_W|\right\}=0 \quad \mbox{ on } \pa\{u_W>0\}.
\end{equation}

\vspace{10pt}

In one dimension, one can rigorously show that the above  heuristics is valid, using the discrete-time (JKO) scheme introduced in \cite{GK} (see also \cite{JKO}):
For a given initial open interval $I^0$ and time step size $h>0$, let us consider the  sequence of open intervals $I^i_h$ with $i=1,2,3,...$, iteratively defined by
\begin{equation}\label{definition101}
I^{i+1}_h := \mbox{argmin}_{I: \mbox{open interval}}  \left\{ \frac{1}{h} \tilde{dist}^2 (I_h^i, I) +\E(I)\right\}, \quad I^0_h = I^0,
\end{equation}
where $\tilde{dist}$ is a modified ``distance"  between open intervals (see \cite{GK} for general formulation and references)
\begin{equation}\label{distance}
\tilde{dist}^2((a,b), (c,d)) = \int_{(a,c) \cup (b,d)} d(x, \{a,b\}) dx.
\end{equation} 

 Let $u_h(\cdot,t)$ and $\lambda_h(\cdot,t)$ be the associated function and lagrange multiplier  to $I^{i+1}_h$ defined at discrete times $t=ih$. Lastly, define 
 $$
 u_h(\cdot,t), \lambda_h(\cdot,t)  \equiv u_h (\cdot,ih), \lambda_h(\cdot,ih)\quad \hbox{ for } t=ih\leq t < (i+1)h.
 $$
 
Following \cite{GK},  we can then show:
\begin{proposition}\label{barrier}
Let $u_h$ and $I_h$ be as given above. 
\begin{itemize}
\item[(a)] Suppose there exists a classical solution $\phi$ of
$$
\left\{
\begin{array}{ll}
-\Delta \phi <\lambda_h(t) -\delta- \phi  \kappa  \cos\alpha + x  \kappa   \sin \alpha \quad & \mbox{ in } \{\phi>0\}\\[5pt]
V<\frac{1}{2} |\na \phi | ^2 -\beta-\delta & \mbox{ on } \pa\{\phi>0\}\cap\{|D\phi|\neq 0\}
\end{array}
\right.
$$
If $h$ is sufficiently small depending on $\delta$, then $\phi$ cannot cross $u_h$ from below.
\item[(b)]  Suppose there exists a classical solution $\phi$ of
$$
\left\{
\begin{array}{ll}
-\Delta \phi >\lambda_h(t) +\delta- \phi  \kappa  \cos\alpha + x  \kappa   \sin \alpha \quad & \mbox{ in } \{\phi>0\}\\[5pt]
V>\frac{1}{2} |\na \phi | ^2 -\beta-\delta & \mbox{ on } \pa\{\phi>0\}
\end{array}
\right.
$$
If $h$ is sufficiently small depending on $\delta$, then $\phi$ cannot cross $u_h$ from above.
\end{itemize}
\end{proposition}

The proof of Proposition \ref{barrier} follows that of Proposition 3.1 and Proposition 3.3 in \cite{GK}.  Note that from \eqref{definition101} it follows that $I_h^i$ cannot move its endpoints by more than $Ch$ from the endpoints of $I_h^{i-1}$, where $C$ depends on the energy associated with $I^0$. It follows that  along a subsequence $u_h, \lambda_h$ and $I_h$ locally uniformly converge as $h\to 0$. Proposition~\ref{barrier} then ensures that the limiting solution satisfies \eqref{eq:uom}  and that the motion law satisfies \eqref{eq:motion3d-2} in the viscosity sense. 

\vspace{10pt}

This discrete-time scheme, which relies on the gradient-flow structure of our problem, can thus be used to prove the existence of solution.
However, it is not a very practical scheme, especially in one dimension. For this reason we will use a different, simpler, discrete-time scheme to construct classical solutions of \eqref{eq:uom}-\eqref{eq:motion3d-2} (see Section~\ref{sec:dt}).

\vspace{10pt}

In the next section, we will state all the results proved in this paper. The rest of the paper (Sections \ref{sec:prel} to \ref{sec:hom}) will be devoted to the proof of these results.

\vspace{10pt}

\section{Main results and outline of the paper}
Throughout the paper, the volume of the drop $V_0$ is fixed (we only compare drops with same volume), and we assume that the drop is invariant in the $y$ direction.
As discuss in the introduction, the sliding drop problem then reduces to finding a function $u(x,t)$ and a set $W(t)=\{u(t)>0\}$ such that
\begin{equation}\label{eq:P} 
\left\{
\begin{array}{l}
\displaystyle u(t) = \mbox{argmin}\{\F(v)\, ;\, v\in H^1_0(W(t)), \;  v\geq 0 ,\; \int v \, dx= V_0 \}\\[5pt]
\displaystyle \min\left\{-V+\frac{1}{2} | u_x | ^2 - \beta\, ,\,  |u_x |\right\}=0 \quad \mbox{ on } \pa\{u>0\}.
\end{array}
\right.
\end{equation}
where $V$ denotes the normal velocity of $\pa W(t)$ and   $\F$ is given by
$$
\F(v) = \int_{\RR} \frac{1}{2} |v_x|^2 + \frac{v^2}{2} \kappa \cos\alpha -v \, x\kappa \sin\alpha\, dx. $$

We note that in one dimension there is no mechanism which would allow a connected drop to split into several droplets (we will show this rigorously in  Section 4, when we construct a solution using a time-discretized scheme).
We thus assume that the drop is connected at time $t=0$, that is $u_0(x)$ satisfies $\{u_0>0\}=(a_0,b_0)$ and 
\begin{equation}\label{eq:init}
u_0 = \mbox{argmin}\{\F(v)\, ;\, v\in H^1_0(a_0,b_0), \;  v\geq 0 ,\; \int v \, dx= V_0 \}.
\end{equation}
Our problem then reduces to finding $W=(a(t),b(t))$ satisfying \eqref{eq:P}.
This leads to the following definition:
\begin{definition}\label{def:1}
Given an  initial data $u_0$ satisfying  \eqref{eq:init}, 
we say that $u(x,t)$ is a solution of  (\ref{eq:P}) if there exist some Lipschitz functions $a(t)$ and $b(t)$ defined for $t\geq 0$,
satisfying
$$ a(0)=a_0,\quad b(0)=b_0$$
and such that
\begin{equation}\label{eq:motionsupp}
  \{x\,;\, u(x,t)>0\} = (a(t),b(t))\quad \mbox{ for all $t\geq 0$,}
\end{equation}
\begin{equation}\label{eq:motionob}
u(\cdot,t)=\mathrm{argmin} \{\F(v)\,;\, v\in H^1_0(a(t),b(t)),\; v\geq 0 ,\; \int_{a(t)}^{b(t)} v(x)\, dx = V_0 \}  \mbox{ for all $t\geq 0$,}
\end{equation}
and 
$$
\left\{
\begin{array}{ll}
\min\{ a'(t) +  \frac{1}{2} |u_x(a(t)) | ^2 - \beta(a(t)), |u_x(a(t))|\} =0 \\[5pt]
\min\{ -b'(t) + \frac{1}{2} |u_x(b(t)) | ^2 - \beta(b(t)) ,|u_x(b(t))|\} = 0
\end{array}
\right.
\quad \mbox{ a.e. $t\geq 0$.}
$$
\end{definition}

With this definition in hand, our first result is the following:
\begin{theorem}\label{thm:ex}
Assume that $x\mapsto \beta(x)$ is a Lipschitz continuous function.
For any $u_0$ satisfying  \eqref{eq:init}, there exists a unique solution $u(x,t)$ in the sense of  Definition \ref{def:1}.
\end{theorem}
Furthermore, we have the following ``comparison principle", which plays a crucial role in our analysis:
\begin{proposition}\label{prop:cp}[Comparison principle and uniqueness]
Assume that $x\mapsto \beta(x)$ is a Lipschitz function, and let $u_1$, $u_2$ be two solutions of  (\ref{eq:P}) (in the sense of  Definition \ref{def:1}) with support
$(a_1(t),b_1(t))$, $(a_2(t),b_2(t))$ such that
\begin{equation}\label{eq:compi} 
a_1(0)\leq a_2(0) \; \mbox{ and }\; b_1(0) \leq b_2(0).
\end{equation}
Then
\begin{equation}\label{eq:compf} 
a_1(t)\leq a_2(t)\;  \mbox{ and } \; b_1(t) \leq b_2(t) \quad \mbox{ for all $t>0$.}
\end{equation}
Furthermore if the inequalities are strict in \eqref{eq:compi}, then they are strict also in \eqref{eq:compf}.
\end{proposition}
We stress out the fact that  the solutions $u_1$ and $u_2$ correspond to drops with the same volume $V_0$ (the result would obviously not hold for two drops with different volume).
The result is not trivial because $u_1$ and $u_2$ do not satisfy the same equation (different Lagrange multipliers).

\vspace{10pt}

Next, we investigate the existence of particular solutions characterizing the asymptotic behavior of general solutions.
We start with the case where $\beta$ is constant, and we show that there exists a unique traveling wave solution (for a given volume and up to translation in time), which is stable:
\begin{theorem}\label{thm:TW}
When $\beta$ is constant, Problem  (\ref{eq:P}) has a  traveling wave type solution of the form $u(x,t)=v(x-ct)$. The profile $v(x)$ is unique (up to translation), and the speed $c$ is given by 
\begin{equation}\label{eq:vitesse}
c=\left\{
\begin{array}{ll}
\frac{1}{2} V_0 \kappa \sin\alpha & \mbox{ if }V_0 \leq 2\frac{\beta}{\kappa \sin\alpha}\\[5pt]
V_0\kappa \sin\alpha - \beta & \mbox{ if }V_0 \geq  2\frac{\beta}{\kappa \sin\alpha}.
\end{array}
\right.
\end{equation}
Finally, any solution of  (\ref{eq:P}) converges as $t\to\infty$ to a translation of this traveling wave $u$.
\end{theorem}

Note that the speed $c$ of the traveling wave is a piecewise linear function of $V_0$, which increases faster once the drop reaches a critical volume (we will see that this correspond to the volume for which the receding contact angle - or the gradient in the rear of the drop - becomes degenerate). Note also that for small volume, the speed is independent of the relative adhesion coefficient $\beta$.  

\vspace{10pt}

We emphasize the fact that when $\beta$ is constant, one can easily show that any drop will be sliding down the inclined plane (no stationary solutions can exist in that case). 
However, the coefficient $\beta$  depends on the properties of the materials (solid and liquid) and  it is very sensitive to small perturbations in the properties of the solid plane (chemical contamination or roughness). 
In the last part of this paper, we thus consider the case of non-constant coefficient $\beta$, and to simplify the analysis, we restrict ourself to periodic settings.

Our goal is to characterize the effects of these heterogeneities on the speed of the sliding droplets.
To this end we consider a simple case in which the constant $\beta$  is replaced by a periodic function
$$ x\mapsto \beta(x).$$

In that case, we will see that stationary solutions do exist, at least for small volume, and that small enough drop may stick to the inclined plane. 
Large volume drops, however, will slide down, and the notion of traveling wave solution is replaced by that of pulsating traveling solutions (which are global in time solutions whose profile exhibits some periodic behavior).
More precisely, we will show:
\begin{theorem}\label{thm:PTF0}
Assume that $x\mapsto \beta(x)$ is a periodic function. Then the following hold:
\item[(i)]  If
$$
\max\beta - \min\beta < V_0\kappa\sin\alpha,
$$
then there exists a unique pulsating traveling solution $u(x,t)$  solution of  (\ref{eq:P}) (that is a global in time solution satisfying $u(x,t+T)=u(x+1,t)$ for all $t$).
Furthermore, any solution will eventually start sliding with positive speed ($b'(t)>0$ for $t\geq t_0$) and will converge to the unique pulsating traveling  solution as $t\to\infty$.
\item[(ii)] If
$$
\max\beta - \min\beta > V_0\kappa\sin\alpha,
$$
then there exists $\delta$ such that if the period of $\beta$ is less than $\delta$, then any solution will stick to the inclined plane. More precisely, we have
\begin{equation}\label{sticking}
\hbox{ either  }a(t)\geq a(0)- C\quad\hbox{ or }b(t)\geq b(0)- C\quad\hbox{ for all }t\geq 0.
\end{equation}
\end{theorem}

Finally, our  last result concerns  the homogenization of the problem above:
We now assume that $x\mapsto \beta_\eps(x)$ is a $\eps$-periodic function (for the sake of simplicity, we take $\beta_\eps(x)=\beta(x/\eps)$).

First, we state the following lemma which will be proved in the appendix:
\begin{lemma}\label{lem:r}
Assume that $x\mapsto \beta(x)$ is periodic. Then,  for all $q\in\RR$, the equation
$$x'(t)=q-\beta(x(t))$$
has a unique global solution (up to translation in time), which is periodic in time.
We define by
\begin{equation}\label{hom:vel} r(q):=\lim_{t\to\infty} \frac{x(t)-x(0)}{t}
\end{equation}
its effective speed. The function $q\mapsto r(q)$ is monotone increasing and continuous, and satisfies
$$r(q)=0, \quad \mbox{ for } q\in [\min \beta,\max \beta].$$
\end{lemma}
Note that the fact that $r(q)$ can be zero in a non trivial interval is a classical and important aspect of the homogenization of contact angle dynamics (see Figure \ref{fig:1}). It typically implies pinning of the contact line: for related results in higher dimensions, we refer to \cite{K} and \cite{K2}.
\begin{figure}[t]
\begin{center}
\scalebox{0.5}{\pdfimage{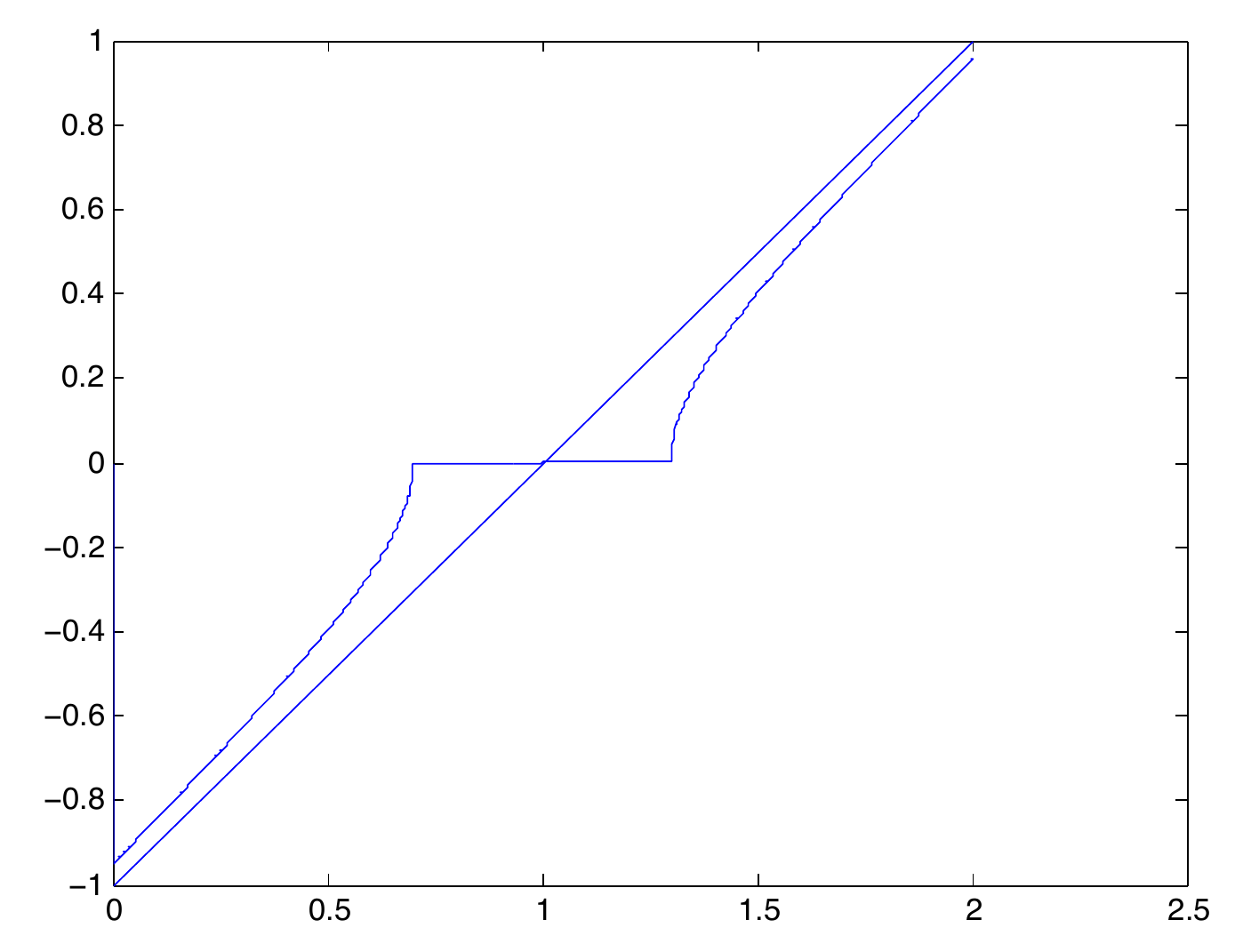}}
\end{center}
\caption{Homogenized velocity function $r(q)$ as a function of the slope $q$ when $\beta$ is given by  $\beta(x)=1+0.3\sin x$.
The straight line is the curve $y=q-\langle \beta \rangle$}
\label{fig:1}
\end{figure}

We can now state our homogenization result:
\begin{theorem}\label{thm:hom}
Let $u^\eps$ be a solution of  (\ref{eq:P}) with initial value $u_0$ and where $\beta=\beta(x/\eps)$ with $\beta$ a periodic function.
When $\eps$ goes to zero, $u^\eps$ converges uniformly to a function $u(x,t)$ solution of \eqref{eq:motionsupp}-\eqref{eq:motionob} satisfying the following velocity law:
\begin{equation}\label{eq:velhom}
\left\{
\begin{array}{ll}
\min\{ a'(t) +  r(\frac{1}{2} |u_x(a(t)) | ^2) , |u_x(a(t))|\} =0 \\[5pt]
\min\{ -b'(t) + r(\frac{1}{2} |u_x(b(t)) | ^2)  ,|u_x(b(t))|\} = 0
\end{array}
\right.
\quad \mbox{ a.e. $t\geq 0$.}
\end{equation}
Furthermore, there is at most one  solution $u$ of the homogenized problem \eqref{eq:motionsupp}-\eqref{eq:motionob}-\eqref{eq:velhom} with given initial data $u_0$.
\end{theorem}

To conclude this analysis, we note that the homogenized problem described in Theorem \ref{thm:hom} has traveling wave solutions for all volume (for small volume, the traveling wave has speed zero, so it is really a stationary solution).
The speed of that traveling wave can be determined as follows:
\begin{itemize}
\item If $r(V_0 \kappa\sin\alpha)+r(0)<0$, then there exists $q_0$ such that 
$$ r(q_0+V_0 \kappa\sin\alpha)+r(q_0)=0$$
and the speed of the traveling wave solution is given by
$$ c=r(q_0+V_0 \kappa\sin\alpha)=-r(q_0)$$
(note that $q_0$ may not be unique, but $c$ is unique).
\item If $r(V_0 \kappa\sin\alpha)+r(0)\geq 0$, then the speed 
of the traveling wave solution is given by
$$ c=r(V_0 \kappa\sin\alpha)$$
(when this happens, the gradient in the rear of the drop of the traveling wave solution is zero)
\end{itemize}
Of course, we recover the speed of the traveling wave for constant coefficient $\beta$ (formula \eqref{eq:vitesse}) when we take $r(q)=q-\beta$.
Figure \ref{fig:2} compares the speed of the traveling wave as a function of $V_0 \kappa \sin\alpha$ for the homogenized  problem \eqref{eq:velhom} when $\beta_\eps(x)=1+0.3\sin (x/\eps)$
 (i.e. with the function $r$ shown in Figure \ref{fig:1}) and for the constant case $\beta=\langle \beta_\eps\rangle=1$.
Note the Lipschitz corner near $ V_0 \kappa \sin\alpha=2$ which corresponds to the critical volume for which the gradient in the rear of the drop becomes zero.

\begin{figure}[t]
\begin{center}
\scalebox{0.5}{\pdfimage{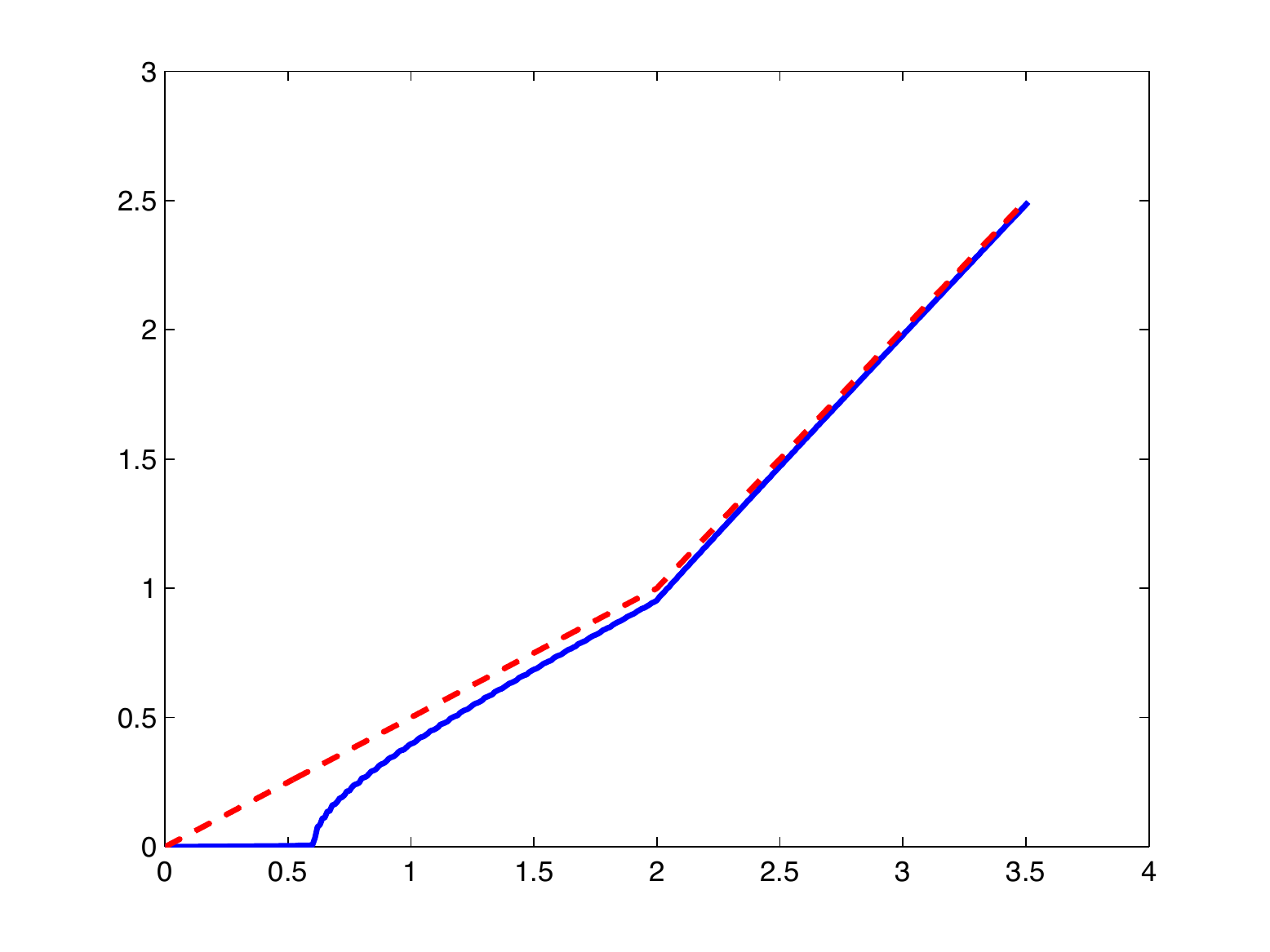}}
\end{center}
\caption{Speed of the traveling wave solution as a function of $V_0 \kappa \sin\alpha$ for the homogenized problem \eqref{eq:velhom} when
$\beta_\eps(x)=1+0.3\sin (x/\eps)$ (solid line), and for the constant case $\beta=\langle \beta_\eps\rangle=1$.}
\label{fig:2}
\end{figure}

\vspace{10pt}

The rest of the paper is devoted to the proof of these results.

\section{Preliminaries}\label{sec:prel}
\subsection{Another formulation for  (\ref{eq:P})}
Definition \ref{def:1} gives the most natural notion of solution for \eqref{eq:P}. 
However, it is not the most practical one. 
In this section, we derive an equivalent formulation which will be more convenient.
First, we note that if $u$ is a solution in the sense of Definition \ref{def:1}, then $u(\cdot ,t)$ solves:
\begin{equation} \label{eq:motionc2}
\left\{
\begin{array}{ll}
-u_{xx} = \lambda(t) - u \kappa \cos\alpha + (x-b(t)) \kappa \sin \alpha \quad & \mbox{ in } (a(t),b(t)) \\[5pt]
u(a(t),t)=u(b(t),t)=0 ,\\[5pt]
\ds \int u\, dx = V_0.
\end{array}
\right.
\end{equation}

It is tempting to replace (\ref{eq:motionsupp})-(\ref{eq:motionob}) by (\ref{eq:motionc2}); However, solutions of (\ref{eq:motionc2}) might take negative values (for large intervals $(a(t),b(t))$).
We thus write:
\begin{definition}\label{def:2} 
Given an  initial data $u_0$ satisfying  \eqref{eq:init}, 
we say that $u(x,t)$ is a solution of   (\ref{eq:P}) if there exist some Lipschitz functions $a(t)$ and $b(t)$ 
satisfying
$$ a(0)=a_0,\quad b(0)=b_0$$
such that $u(\cdot,t)$ solves (\ref{eq:motionc2}) for all $t\geq 0$ and  $a(t)$, $b(t)$ satisfy
\begin{equation} \label{eq:motionc1}
\left\{
\begin{array}{ll}
\min\{ a'(t) +  \frac{1}{2} |u_x(a(t)) | ^2 - \beta(a(t)), u_x(a(t))\} =0 \\[5pt]
b'(t) = \frac{1}{2} |u_x(b(t)) | ^2 - \beta(b(t)) .
\end{array}
\right.
\end{equation}
\end{definition}

In this second formulation, the positivity  of $u$ in $(a(t),b(t))$ is enforced by the condition $u_x(a)\geq 0$. Indeed, we can show:
\begin{lemma}\label{lem:pos}
Let $u(\cdot,t)$ be the solution of (\ref{eq:motionc2}), then $-u_x(b(t),t)>0$. Furthermore, $u(x,t)>0$ in $(a(t),b(t))$ if and only $u_x(a(t),t)\geq 0$.
In particular, Definitions \ref{def:1} and \ref{def:2} are equivalent.
\end{lemma}
Note that it might have seemed more natural to write the second equation in  \eqref{eq:motionc1} as
$$\min\{ -b'(t) + \frac{1}{2} |u_x(b(t)) | ^2 - \beta(b(t)) ,-u_x(b(t))\}.$$
However, lemma \ref{lem:pos} shows that the condition $-u_x(b(t))>0$  always holds.

\begin{proof}[Proof of Lemma \ref{lem:pos}]
We drop the $t$ dependence in what follows.
Since $V_0>0$, there exists $x_0\in(a,b)$ such that $u(x_0)=\max_{x\in (a,b)} u>0$. Since $-u''(x_0)\geq 0$,  (\ref{eq:motionc2}) implies
\begin{equation}\label{eq:lambda}
 \lambda +( x_0-b)\kappa \sin \alpha   \geq  u(x_0) \kappa \cos\alpha >0 .
\end{equation}
Assume now that  $u$ has a local minimum at a point $x_1\in(a,b)$. Then $-u''(x_1)\leq 0$, and so  (\ref{eq:motionc2}) implies
$$ \lambda + (x_1-b)\kappa \sin \alpha   \leq  u(x_1) \kappa \cos\alpha .
$$ 
Since $u(x_1)\leq v(x_0)$, we deduce 
$$ x_1 \leq x_0,$$ 
and so we always have $u(x)> 0$ in $(x_0,b)$. In particular, we have $-u'(b)\geq 0$.
In order to show that $-u'(b)> 0$, we note that (\ref{eq:lambda}) also implies
 $\lambda  >0$, and so the function $h(x)=-\mu (x-b)$ satisfies
 $$ -h'' \leq \lambda - h \kappa \cos\alpha + (x-b) \kappa \sin \alpha = \lambda -[\mu  \kappa \cos\alpha +  \kappa \sin \alpha](b-x) $$
 in a neighborhood of $b$. A Hopf's Lemma - type argument yields $-u'(b)> 0$.
\medskip

Assume now that $u'(a)\geq 0$, and  that there exists $x_0 \in (a,b)$ such that $u(x_0)\leq 0$.
Using the boundary conditions, we deduce that there exist at least three distinct points in $[a,b]$ where $u'$ vanishes, and so there exists at least two distinct points $y_1$, $y_2$ in $(a,b)$ where $u''$ vanishes.

Therefore, the function $w= u \kappa \cos\alpha -\lambda-(x-b) \kappa \sin \alpha $, which solves 
$$w''=\kappa \cos\alpha w \quad \mbox{ in } (a,b)$$
is such that $w''(y_i)=w(y_i)=0$, $i=1,\, 2$ and the maximum principle implies $w=0$ in $(y_1,y_2)$ and thus $w=0$ in $(a,b)$, which is impossible (look at the values of $w$ at $a$ and $b$).

We have thus shown that $u'(a)\geq 0$ implies that $u>0$ in $(a,b)$.
Conversely, if $u>0$ in $(a,b)$, it is readily seen that $u'(a)\geq 0$.

\end{proof}

Finally, we give a third definition which will be very useful later on.  We will show in  Proposition \ref{prop:obstacle} $(iv)$ that there exists a {\bf critical length} $\ell_c>0$ (depending on the volume $V_0$ and the inclination $\alpha$) such that the solution of (\ref{eq:motionc2}) is positive in $(a(t),b(t))$ if and only if $b(t)-a(t)\leq \ell_c$. 
This length $\ell_c$ is the longest possible length of the support of the drop.
This leads to the following  third formulation:

\begin{definition}\label{def:3}
Given an  initial data $u_0$ satisfying  \eqref{eq:init}, 
we say that $u(x,t)$ is a solution of  (\ref{eq:P}) if  there exist some Lipschitz functions $a(t)$ and $b(t)$ 
such that $u(\cdot ,t)$ solves (\ref{eq:motionc2})
and 
$$ a(0)=a_0,\quad b(0)=b_0$$
with
\begin{equation} \label{eq:mmotionc1'}
\left\{
\begin{array}{ll}
\min\{ a'(t) +  \frac{1}{2} |u_x(a(t)) | ^2 -\beta(a(t)),a(t)-b(t)+\ell_c\}=0 \\[5pt]
b'(t) = \frac{1}{2} |u_x(b(t)) | ^2 - \beta(b(t)) ,
\end{array}
\right.
\end{equation}
a.e. $t\geq 0$.
\end{definition}

Equivalently, denoting $\ell(t)=b(t)-a(t)$, we can write:
\begin{equation} \label{eq:motionl}
\left\{
\begin{array}{ll}
 \max\left\{ \ell'(t) -\left[ \frac{1}{2} |u_x(b(t)) | ^2+ \frac{1}{2} |u_x(a(t)) | ^2 - \beta(a(t))-\beta(b(t)) \right] , \ell(t)-\ell_c\right\} = 0
  \\[5pt]
b'(t) = \frac{1}{2} |u_x(b(t)) | ^2 - \beta(b(t)) ,
\end{array}
\right.
\end{equation}
where $u_x(b(t))$ and $u_x(a(t))$ only depends on $l(t)$ (see Proposition \ref{prop:obstacle} $(iii)$).

We then have:
\begin{proposition}\label{prop:eq}
Definition \ref{def:2} and \ref{def:3} are equivalent.
\end{proposition}

In order to prove this proposition, we will need to derive several important properties of the solution of the obstacle problem \eqref{eq:motionob}, which is done in the next section.
Proposition \ref{prop:eq} will be an immediate consequence of Proposition \ref{prop:obu} below.

\subsection{Analysis of the obstacle problem \eqref{eq:motionob}}\label{sec:ob}
Throughout this section, we fix a volume $V_0$, and all constants will depends on $V_0$, $\alpha$ and $\kappa$.

Given $a<b$, we consider  $u$ solution of 
\begin{equation}\label{eq:obstacle}
u=\mathrm{argmin} \{\F(v)\,;\, v\in H^1_0(a,b),\; v\geq 0 ,\; \int_a^b v(x)\, dx = V_0 \}.
\end{equation}
Then classical arguments for obstacle type problems yield that $u\in \mathcal C^{1,1}(a,b)$ solves
\begin{equation}\label{eq:obmin}
\min\{ -u''- \lambda - (x-b) \kappa \sin \alpha + u \kappa \cos\alpha, u\} =0 \mbox{ in } (a,b), \quad u(a)=u(b)=0, \quad \int u \, dx= V_0 
\end{equation}
(where $\lambda\in \RR$ is determined by the volume  constraint).

The main properties of $u$ are gathered in the following proposition:
\begin{proposition}\label{prop:obstacle}
Let $u$ be a $\mathcal C^{1,1}$ solution of (\ref{eq:obmin}), then:
\item[(i)] $\{u>0\}$ is an interval of the form $(a',b)$  and $-u'(b)> 0$.
\item[(ii)]  $u$ is the unique solution of (\ref{eq:obmin}), and is also solution of (\ref{eq:obstacle}).
\item[(iii)] $\lambda$,   $-|\{u>0\}|$, $-u'(b)$ and $u'(a')$ only depends on $a$ and $b$ through $(b-a)$ and are decreasing functions of $b-a$.
\item[(iv)] There exists $\ell _c>0$ such that $\{u>0\}=(a,b)$ if and only if $b-a\leq \ell_c$.
Furthermore, $u'(a)>0$ if and only if $b-a<\ell_c$.
\end{proposition}
The critical length $\ell_c$ is the longest possible length of the support of the drop. We recall that it depends on $V_0$, $\kappa$ and $\alpha$. It will play a very important role in the analysis of our problem.
 
 \medskip
In view of $(iii)$, we can define
\begin{equation}\label{not:GH} 
G(b-a):= \frac{1}{2} |u'(a)|^2,  \quad H(b-a):=  \frac{1}{2} |u'(b)|^2 
\end{equation}
where $u$ is the solution of  (\ref{eq:obmin}). 
The functions $\ell\mapsto G(\ell)$ and  $\ell\mapsto H(\ell)$ are then both monotone decreasing, and multiplying the equation (\ref{eq:obmin}) by $u'$ and integrating over $(a',b)$ yields
$$
 \frac{1}{2}u'(b)^2 -\frac{1}{2} u'(a')^2 =  \kappa \sin\alpha V_0 
$$
or
\begin{equation}\label{eq:GHV}
 H(\ell)-G(\ell) = 
\kappa \sin\alpha V_0 .
\end{equation}

We also define 
\begin{equation}\label{eq:Fdef}
F: \ell \mapsto G(\ell)+H(\ell),
\end{equation}
which is a decreasing function as well.

Finally, we have the following lemmas, which will be useful later on:
\begin{lemma}[Lipschitz regularity]\label{lem:H}
For all $\delta_0>0$, there exists $C>0$ such that
$$0\leq  H(\ell_0)-  H(\ell_0+\eta)\leq C\eta$$ and
$$0\leq G(\ell_0)- G(\ell_0+\eta) \leq C\eta$$
for $\ell_0\geq \delta_0$ and $\eta\geq 0$.
In particular, it follows that $H$ and $G$ (and therefore $F$) are locally Lipschitz on $(0,\infty)$.
\end{lemma}
 
\begin{lemma}[Non-degeneracy]\label{lem:ndgH}
For all $\delta_0>0$ there exists $c>0$ such that
\begin{equation}\label{eq:ndgH}  
H(\ell_0) -  H(\ell_0+\eta)\geq c\eta
\end{equation}
for all $\ell_0\leq \ell_c-\delta_0$ and
 $\eta$ such that $\ell_0+\eta \leq \ell_c$
 \end{lemma}

Next, we consider the following boundary value problem with volume constraint:
\begin{equation}\label{eq:obu}
\left\{
\begin{array}{l}
 -v'' = \lambda - v \kappa \cos\alpha +( x-b) \kappa \sin \alpha \quad  \mbox{ in } (a,b) \\[5pt]
\mbox{$v(a)=v(b)=0$ and $\int v\, dx =V_0$.}
\end{array}
\right.
\end{equation}
where $\lambda$ is a Lagrange multiplier for the volume constraint.
Then, we have:
\begin{proposition}\label{prop:obu}
Let $v$ be the solution of (\ref{eq:obu}).
Then there exists $\ell_c$ depending on $V_0$, $\alpha$ and $\kappa$ such that
 the followings are equivalent
\item[(i)] $v'(a)\geq 0$
\item[(ii)] $v\geq 0$ is also  the solution of (\ref{eq:obstacle})
\item[(ii)] $b-a\leq \ell_c$
\item Furthermore, $v'(a)>0$ if and only if $b-a<\ell_c$.
\end{proposition}
This proposition implies Proposition \ref{prop:eq}.

The remainder of this section is devoted to the proof of these results.

\begin{proof}[Proof of Proposition \ref{prop:obstacle}]
\item[(i)] Let $(c,d)$ be a connected component of $\{u>0\}$, then
\begin{equation}\label{eq:uu1}
  -u'' = \lambda + (x-b) \kappa \sin \alpha - u \kappa \cos\alpha 
\end{equation}
in $(c,d)$.
Multiplying (\ref{eq:uu1}) by $u'$ and integrating over $(c,d)$ yields:
$$ u'(d)^2 = u'(c)^2 + 2 \kappa \sin\alpha \int_c^d u\, dx > 0.$$
If $d\neq b$, then the regularity of $u$ implies $u'(d)^2 =0$, a contradiction. So $d=b$, $|u'(b)| >0$ and $(c,b)$ is the only connected component of $\{u>0\}$. We also have the following formula:
\begin{equation}\label{eq:a'}
 u'(b)^2 - u'(a')^2 = 2 \kappa \sin\alpha V_0 
 \end{equation}

 \medskip 
 
\item[(ii)] To prove the uniqueness of the solution of (\ref{eq:obmin}), we assume that $u_1$, $u_2$ are two solutions of (\ref{eq:obmin}) (with $\lambda_1$ and $
\lambda_2$ respectively).
We know that $\{u_i>0\}=(a'_i,b)$ and we can assume (without loss of generality) that $a\leq a_2'\leq a'_1$.

If $v = u_1-u_2$, then
$$ -v'' = \lambda_1-\lambda_2 - v \kappa \cos\alpha \quad \mbox{ in } (a'_1,b).$$
Furthermore, $v(a'_1)\leq 0$, $v(b)=0$ and 
$\int_{a'_1}^b v \, dx > 0$. So  $v$ reaches its positive maximum value somewhere in $(a_1',b)$. We deduce
$$ \lambda_1\geq \lambda_2.$$
 
Next, we note that since 
$$  -u_1'' \geq  \lambda_1 + (x-b) \kappa \sin \alpha - u_1 \kappa \cos\alpha\quad \mbox{ in } \{a'_2,b\},$$ 
we also have 
$$ -v'' \geq  \lambda_1-\lambda_2 - v \kappa \cos\alpha \geq - v \kappa \cos\alpha \quad \mbox{ in } (a'_2,b),$$
with $v(a'_2)=v(b)=0$ and $\int _{a'_2}^{b} v\, dx=0$. 
If $v\neq 0$, then $v$ has a (strictly) negative minimum value in $ (a'_2,b)$, a contradiction.
We deduce $v=0$ and so $u_1=u_2$.  

\medskip

\item[(iii)] Equation (\ref{eq:obmin}) is invariant by translation, and thus the uniqueness obtained in (ii) implies that $u(\cdot +b)$ and $\lambda$ only depends on $\ell=(b-a)$ (in particular, $u'(a)$ and $-u'(b)$ only depends on $\ell$).

The proof of (ii) above also implies that $\lambda$ is a decreasing function of $(b-a)$. 
We now take $a_2<a_1<b$ and let
 $u_1$ (resp. $u_2$) be the solution of (\ref{eq:obmin}) on the interval $(a_1,b)$ (resp. $(a_2,b)$).
The function $v=u_1-u_2$ satisfies
$$v(a_1)\leq 0, \quad v(b)=0, \mbox{ and } \int_{a_1}^b v \geq 0$$ 
so there exists a point $x_0\in [a_1,b)$ such that $v(x_0)=0$. 
Since
$$ -v'' + v \kappa \cos\alpha= \lambda_1-\lambda_2 \geq 0  \quad \mbox{ in } (x_0,b),$$
the maximum principle gives $ v \geq 0$ in $(x_0,b)$. In particular, $-v'(b)\geq 0$ and so
$$ -u_1'(b) \geq -u_2'(b).$$

Finally, if $u_2'(a_2)=0$, then $ u_2'(a_2) \leq u_1'(a_1)$, while if $ u_2'(a_2) >0$, then (\ref{eq:a'}) implies
$$ u_2'(a_2)^2 =u_2'(b)^2 - 2 \kappa \sin\alpha V_0 
$$
and 
$$ u_1'(a'_1)^2 =u_1'(b)^2 - 2 \kappa \sin\alpha V_0 \geq u_2'(b)^2 - 2 \kappa \sin\alpha V_0  = u_2'(a_2)^2
$$ 
hence $a'_1=a_1$ and $u_1'(a_1)\geq u_2'(a_2)$.

\medskip

\item[(iv)]  Note that if $\{u>0\}\neq (a,b)$, then $u$ is the unique solution of (\ref{eq:obmin}) in any interval larger than $(a,b)$.
So we can define 
$$\ell_c:= \min \{ \ell\; ; \mbox{ the solution of (\ref{eq:obmin}) in $(0,\ell)$ satisfies } \{u>0\}\neq (0,\ell) \}$$
We only need to show that this set is non empty (that is $\ell_c<\infty$), and that this is equivalent to $u'(a)\neq 0$.

Assume that $ \{u>0\}\neq (a,b)$. Then
$$ -u''\leq \lambda +(x-b) \kappa \sin\alpha\qquad \mbox{ in } (a,b)
$$
and so
$$ -u'(b)+u'(a) \leq  \lambda(b-a) - \frac{1}{2}(b-a)^2  \kappa \sin\alpha $$
Since $u'(a)\geq 0$ and $u'(b)\leq 0$, we deduce
$$(b-a) \kappa \sin\alpha \leq \lambda$$
Since $\lambda$ is a decreasing function of $(b-a)$, this implies that  $(b-a)\leq C$ for some constant $C$.

\medskip

Finally, we note that if $b-a<\ell_c$, then $u'(a)>0$, while the regularity of $u$ implies $u'(a)=0$ when  $b-a\geq \ell_c$.

\end{proof}

We now turn to the proof of the Lipschitz regularity of the functions $G$ and $H$: 
\begin{proof}[Proof of Lemma \ref{lem:H}] 
Note that $H$ and $G$ are constant for $\ell\geq \ell_c$, so we can always assume that 
\begin{equation}\label{eq:ellceta}
\delta_0\leq \ell_0\leq\ell_0+\eta\leq \ell_c.
\end{equation}

We now consider the solution $u$ of \eqref{eq:obmin} with $a=0$ and $b = \ell_0$ and
we denote by $\lambda(\ell_0)$ the Lagrange multiplier appearing in \eqref{eq:obmin}. We recall that $\lambda$ is a decreasing function of $\ell_0$ and we  claim that 
\begin{equation}\label{claim1}
\lambda(\ell_0)-\lambda(\ell_0+\eta) \leq C\eta
\end{equation} 
for some constant $C$ depending on $\delta_0$. 
To prove this, we consider $v$ solution of the same equation as $u$ (with the same $\lambda$) but on the interval $(0, \ell_0+\eta)$ instead of $(0,\ell_0)$:
$$ -v''= \lambda(\ell_0) + (x-\ell_0) \kappa \sin \alpha - v \kappa \cos\alpha\quad \mbox{ in } (0,\ell_0+\eta), \quad v(0)=v(\ell_0+\eta)=0.
$$
We note that $v\geq u$ and that the function $h=v-u$ solves
$$
h\kappa\cos\alpha-h'' = 0 \hbox{ on } (0,\ell_0).
$$
In particular, on the interval $(0,\ell_0)$,  $h$ must take  its positive maximum at $x=\ell_0$. Now on $(\ell_0,\ell_0+\eta)$ we have $h=v\leq C_1 \eta$  for some $C_1>0$ with $C_1$ depending on $\delta_0$. We conclude that $h \leq C_1\eta$ in $(0,\ell_0+\eta)$ and thus  
\begin{equation}\label{eq:Vv}
\int v \, dx \leq V_0+C_1' \eta
\end{equation}
for some constant $C'_1$ depending on $\delta_0$ and $\ell_c$ (recall that $\ell_0+\eta\leq \ell_c$).
\vspace{10pt}

Next, we consider $w$ the solution of \eqref{eq:obmin} with $a=0$ and $b = \ell_0+\eta$ (and so  $\lambda=\lambda(\ell_0+\eta)$).
The function $g=v-w$ solves
$$ 
 g \kappa \cos\alpha-g'' 
=\sigma 
$$
 in $(0,\ell_0+\eta)$, with $g(0)=g(\ell_0+\eta)=0$, where (since $\lambda$ is decreasing)
 $$ \sigma =  \lambda(\ell_0)-\lambda(\ell_0+\eta) + \eta  \kappa \sin \alpha > 0.$$

Finally, we note that the solution of $-\phi''=\mu \sigma$ in $(0,\ell_0+\eta)$ with zero boundary condition satisfies $\phi\kappa\cos\alpha-\phi'' \leq \sigma  $, if $\mu$ is small enough. 
The comparison principle thus implies that $g \geq \phi$ and so
$$\int g dx = \int v\, dx - V_0 \geq \int \phi\, dx \geq C \mu \sigma$$
 for some small $\mu>0$.
Equation \eqref{eq:Vv} implies
$$ C_1' \eta \geq C\mu\sigma$$
and so
$$ \lambda(\ell_0)-\lambda(\ell_0+\eta) \leq C\eta$$
which proves (\ref{claim1}).

\vspace{10pt}

We can now prove the Lipschitz bounds for the functions $H$ and $G$ by
comparing $u$ (solution of \eqref{eq:obmin}  on $(0,\ell_0)$) and $w$   (solution of \eqref{eq:obmin}  on $(0,\ell_0+\eta)$):
The function $q(x)=u(x)-w(x)$ solves
$$ q\kappa\cos\alpha  -q'' = \lambda(\ell_0)-\lambda(\ell_0+\eta) +\eta \kappa\sin\alpha , \quad \mbox{ in } (0,\ell_0)$$
with $q(0)=0$ and $q(\ell_0)=-w(\ell_0)\leq 0$.
So \eqref{claim1} implies
$$q \kappa\cos\alpha  -q'' \leq C\eta  \quad \mbox{ in } (0,\ell_0), \quad  \mbox{ with } q(0)=0 \mbox{ and } q(\ell_0)\leq 0.$$
We deduce
$ q'(0)\leq C\eta$ and so 
$$ u'(0)\leq w'(0)+C\eta$$
which proves our statement for $G$. The proof for $H$ is similar (fixing $b=0$ and choosing $a=-\ell_0$ or $a=-\ell_0-\eta$).
\end{proof}

The proof of Lemma \ref{lem:ndgH} follows from similar argument:
\begin{proof}[Proof of Lemma \ref{lem:ndgH}]
We recall that $\ell_0\leq \ell_c-\delta_0$ and  
we fix $b=0$, $a_2 = -(\ell_0+\eta)$ and $a_1 = -\ell_0$.
We note that it is enough to prove \eqref{eq:ndgH} for $\eta\leq (\ell_c-\ell_0)/2$, so that we can assume that $\ell_0+\eta \leq \ell_c-\delta_0/2$.

We now consider
 $u_1$ (resp. $u_2$) the solution of (\ref{eq:obmin}) on the interval $(a_1,0)$ (resp. $(a_2,0)$).
The function $h=u_1-u_2$ satisfies in particular
$$h(a_1)\leq 0, \quad h(0)=0 $$
and we are going to show that $\max_{(a_1,0)} h \geq c\eta$ for some $c>0$.
\medskip

First of all, the volume constraint   implies that 
\begin{equation}\label{balance}
 \int_{a_1}^{b} h \, dx = \int_{a_2}^{a_1}u_2\, dx>0
\end{equation}
Furthermore, since $\ell_0+\eta< \ell_c-\delta_0/2$ the monotonicity of $G$ implies
$$u_2'(a_2)>\sqrt{2G(\ell_c-\delta_0/2)}>0$$
and so $u_2(a_1)\geq c\eta$ for some $c>0$.
Also, Lemma~\ref{lem:H} gives
\begin{eqnarray*}
 h'(a_1) & = & u_1'(a_1) -u_2'(a_1) \\
 & = & G(\ell_0) -   u_2'(a_1)\\
 & \leq & G(\ell_0+\eta) + C\eta-   u_2'(a_1)\\
 & \leq &  C\eta+ u_2'(a_2)-   u_2'(a_1)\\
 & \leq & C\eta
 \end{eqnarray*}
 We deduce that there exists $c>0$ such that 
\begin{equation}\label{eq:ha1}
h(a_1)=-u_2(a_1)\leq -c\eta \quad \mbox{ and } h'(a_1)\leq c\eta.
\end{equation}
Finally, we have $ -h''+h \kappa \cos\alpha  = \lambda_1-\lambda_2$ in  $(a_1,0),$ and so \eqref{claim1} implies
\begin{equation}\label{eq:ha10} 
 0\leq -h''+h \kappa \cos\alpha   \leq C\eta.
\end{equation}

Consider now $g$ solution of $  -g''+g \kappa \cos\alpha=C$ on $(a_1,0)$ with $g(a_1)=-c$ and $g'(a_1)=c$. 
Then  (\ref{eq:ha1}) and \eqref{eq:ha10} imply that $h\leq \eta g$ in $(a_1,0)$.

Let now $y_0\in(a_1,0)$ be such that $g<0$ in $(a_1,y_0)$ and $g(y_0)=0$ ($y_0$ only depends on the constant $c$ and $C$ appearing in \eqref{eq:ha1} and \eqref{eq:ha10}), 
then
\begin{equation}\label{counterbalance}
\int_{a_1}^{y_0} h\, dx  \leq \eta \int_{a_1}^{y_0} g\, dx  \leq -c\eta
\end{equation}
with $c>0$.

Equations \eqref{balance} and \eqref{counterbalance} imply that 
$$\int_{y_0}^{0} h(x) dx\geq c\eta $$ 
and so  there exists a point $z_0\in (y_0,0)$ where $h(z_0) \geq c\eta$. 
A Hopf's lemma type argument (using \eqref{eq:ha10}) now yields $|h'(0)| >c \eta$
and so 
$$
|u_1'(0)| > |u_2'(0)| +c \eta.
$$
The result follows. 
 \end{proof}

\begin{proof}[Proof of Proposition \ref{prop:obu}]
\item[\bf (i)$\Rightarrow$(ii)]
If $v'(a)\geq 0$, then Lemma \ref{lem:pos} implies that $v>0$ in $(a,b)$, and so $v$ is also the solution of (\ref{eq:obmin}), and thus of (\ref{eq:obstacle}).

\item[\bf (ii)$\Rightarrow$(iii)]
If $v$ is solution of (\ref{eq:obu}), then $v\geq 0$ in $(a,b)$, and since $v$ cannot be zero a an interval of $(a,b)$, we must have $\{v>0\}=(a,b)$. Proposition \ref{prop:obstacle} (iv) gives $b-a\leq \ell_c$.

\item[\bf (iii)$\Rightarrow$(i)]
Finally, if  $b-a\leq \ell_c$, then the solution $u$ of (\ref{eq:obstacle}) satisfies $\{u>0\}=(a,b)$, and is thus also a solution of (\ref{eq:obu}). If follows that $v'(a)=u'(a)\geq 0$.

\end{proof}

\section{Proof of Theorem \ref{thm:ex}}
\subsection{Uniqueness and Comparison principle}
In this section, we prove the comparison principle, Proposition \ref{prop:cp}, which also implies the uniqueness part of Theorem \ref{thm:ex}:

\begin{proof}[Proof of Proposition \ref{prop:cp}]
Thanks to Proposition \ref{prop:eq}, we can assume that $u$ is a solution in the sense of Definition \ref{def:3}.
Using the notations of (\ref{not:GH}), we can rewrite \eqref{eq:mmotionc1'} as follows
\begin{equation}\label{eq:GH}
\left\{
\begin{array}{ll}
\min\{ a'(t)+G(\ell(t))-\beta(a(t)),\ell_c-\ell(t)\}=0 \\[5pt]
b'(t) = H(\ell(t))-\beta(b(t)) ,
\end{array}
\right.
\end{equation}
with $\ell\mapsto H(\ell)$ and $\ell\mapsto G(\ell)$ monotone decreasing.
We now define 
$$ f(t) =\max\{a_1(t)-a_2(t),b_1(t)-b_2(t)\}.$$
Since the functions $a_i(t)$ and $b_i(t)$ are in $W^{1,\infty}_{loc}(0,\infty)$, $f(t)$ is in $W^{1,\infty}_{loc}(0,\infty)$ and we have
$$ f'(t)=
\left\{
\begin{array}{ll}
a'_1(t)-a'_2(t) & \mbox{ if } a_1(t)-a_2(t)>b_1(t)-b_2(t) \\[5pt]
b'_1(t)-b'_2(t) & \mbox{ if } a_1(t)-a_2(t)< b_1(t)-b_2(t) \\[5pt]
b'_1(t)-b'_2(t) = a'_1(t)-a'_2(t)  & \mbox{ if } a_1(t)-a_2(t)= b_1(t)-b_2(t) 
\end{array}
\right. \qquad \mbox{a.e. } t\geq 0
$$
and so 
$$ f'(t)=
\left\{
\begin{array}{ll}
a'_1(t)-a'_2(t) & \mbox{ if } \ell_2 (t)>\ell_1(t) \\[5pt]
b'_1(t)-b'_2(t) & \mbox{ if } \ell_2 (t)\leq \ell_1(t)
\end{array}
\right.
$$
When $\ell_2 (t)>\ell_1(t)$, we have in particular $\ell_1(t)<\ell_c$, and so (\ref{eq:GH}) gives
$$ a'_1(t) = -G(\ell_1(t))+\beta(a_1(t)), \quad a'_2(t) \geq -G(\ell_2(t))+\beta(a_2(t)).$$
We deduce
$$
 f'(t)\leq 
\left\{
\begin{array}{ll}
G(\ell_2(t))-G(\ell_1(t)) +\beta(a_1(t))-\beta(a_2(t))& \mbox{ if } \ell_2 (t)>\ell_1(t) \\[5pt]
H(\ell_1(t))-H(\ell_2(t)) -\beta(b_1(t))+\beta(b_2(t)) & \mbox{ if } \ell_2 (t)\leq \ell_1(t)
\end{array}
\right.
$$
and the monotonicity of $G$ and $H$ implies 
$$
 f'(t)\leq 
\left\{
\begin{array}{ll}
K|a_1(t)-a_2(t)| & \mbox{ if } a_1(t)-a_2(t)>b_1(t)-b_2(t) \\[5pt]
K|b_1(t)-b_2(t)| & \mbox{ if } a_1(t)-a_2(t)\leq b_1(t)-b_2(t) 
\end{array}
\right.
$$
with $K=||\beta'||_{L^\infty}$.
We thus have 
$$ f'(t) \leq K |f(t)| \quad \mbox{  a.e.  $t\geq 0$.}$$
In particular, $f_+'(t)\leq K f_+(t)$ and so $f_+(t) \leq f_+(0) e^{Kt}$,
so if $f(0)\leq 0$, then $f(t)\leq 0$ a.e. $t\geq 0$.

\vspace{10pt}

We also have $f_-'(t)\geq -Kf_-'(t)$ and so $f_-(t)\geq f_-(0)e^{-Kt}$, so if $f(0)< 0$, then $f(t)< 0$ a.e. $t\geq 0$. The result follows.
\end{proof}

\subsection{Existence of a solution}\label{sec:dt}

Observe that, when $\beta$ does not depend on $x$,  formulation (\ref{eq:motionl}) implies that the problem can be reduced to solving first an equation for $\ell(t)$, of the form
\begin{equation}\label{eq:eql}
 \max\{ \ell'(t)-[F(\ell(t))-2\beta],\ell(t)-\ell_c\}=0\}
 \end{equation}
 with $F$ defined by \eqref{eq:Fdef},
and then an equation for $b(t)$ of the form
\begin{equation}\label{eq:eqb}
 b'(t)=H(\ell (t))-\beta.
 \end{equation}
When $\beta$ depends on $x$ however, one cannot decouple the equations for $a$ and $b$.  We will thus use a discrete scheme to prove the existence of solutions in this general framework.

\subsubsection{Discrete-time scheme}
To prove the existence of a solution for general coefficient $\beta$, we introduce a simple discrete time scheme (simpler than the gradient flow scheme described in the introduction):
For $h >0$ small,  we describe the evolution of the drop as follows:

Assume that the support of the drop at time $t_n=nh $ is the interval $(a_n,b_n)$ and denote by $u_n(x)$ the corresponding solution of (\ref{eq:obstacle}).
We define 
$$
\begin{array}{l} 
a_{n+1/2} = a_{n} + \Big[\beta(a_n)-  \frac{1}{2} |u'_n(a_n) | ^2\Big] h \\[5pt]
b_{n+1/2} = b_n + \Big[\frac{1}{2} |u'_n(b_n) | ^2 - \beta(b_n) \Big] h 
\end{array}
$$
Next, we define $u_{n+1}$ as the solution of the obstacle problem (\ref{eq:obstacle}) with $a=a_{n+1/2}$ and $b=b_{n+1/2}$.
Proposition \ref{prop:obstacle} (i) implies that there exists 
 $a_{n+1}$ and $b_{n+1}=b_{n+1/2}$ such that
$$ \{u_{n+1}>0\} = (a_{n+1},b_{n+1}).$$

This scheme defines a sequence of functions $\{u_n(x)\}_{n\in\NN}$ and a sequence of intervals $\{(a_n,b_n)\}_{n\in \NN}$.
We then define $u^h(x,t)$, $a^h(t)$ and $b^h(t)$ continuous piecewise linear functions such that
$$ u^h(x,nh) := u_n(x)$$
$$ a^h(nh) := a_n,\;  b^h(nh) := b_n.$$
Our goal is now to pass to the limit $h\to0$.
\medskip

First, we observe that since $ b_{n+1}=b_{n+1/2}$, we have
\begin{equation}\label{eq:md1}
\frac{b_{n+1} -b_n }{h } = \frac{1}{2} |u'_n(b_n) | ^2 - \beta(b_n) .
\end{equation}
On the other hand,  we only have $a_{n+1}\geq a_{n+1/2}$ and if $a_{n+1}>a_{n+1/2}$, then
$ u'_{n+1}(a_{n+1}) = 0 .$ So we can write
$$
\frac{a_{n+1} - a_{n}}{h } \geq \beta(a_n)-  \frac{1}{2} |u'_n(a_n) | ^2 ,\qquad
\mbox{ with equality if $u'_{n+1}(a_{n+1}) >0 $,}$$ 
which implies 
\begin{equation}\label{eq:md2}
\min\left\{ \frac{a_{n+1} - a_{n}}{h }+\frac{1}{2} |u'_n(a_n) | ^2  - \beta(a_n),   u'_{n+1}(a_{n+1}) \right\} =0.
\end{equation}
Finally, using Proposition \ref{prop:obu}, we can also write: 
\begin{equation}\label{eq:md2'}
\min\left\{ \frac{a_{n+1} - a_{n}}{h }+\frac{1}{2} |u'_n(a_n) | ^2  - \beta(a_n),  a_n-b_n+\ell_c   \right\} =0.
\end{equation}

\subsubsection{Limit $h\to0$}
We now study the limit $h \rightarrow 0$ of the discrete model introduced above.
Note that $u_n$ satisfies, for some $\lambda_n$,
\begin{equation} \label{eq:un}
\left\{
\begin{array}{ll}
-u''_{n} = \lambda_n - u_n \kappa \cos\alpha + (x-b_n) \kappa \sin \alpha \quad& \mbox{ in } (a_{n},b_{n}),\\[5pt]
u_n(a_n)=u_n(b_n)=0 ,\\[5pt]
\int u_n\, dx = V_0.
\end{array}
\right.
\end{equation}

We then have the following proposition:
\begin{proposition}\label{prop:properties}
For all $n$, the following holds:
\begin{itemize}
\item[(i)] $|u'_n(a_n)|$ and $|u'_n(b_n)|$ are uniformly bounded. More precisely, there exists a constant $M$ depending only on $\kappa$, $\alpha$, $V_0$ and the initial data such that
$$
|u_n'(a_n)|\leq |u_n'(b_n)| \leq M.
$$ 
\item[(ii)] There exist constants $\underline\ell>0$ and $\Lambda$ depending only on $\kappa$, $\alpha$, $V_0$ such that $b_n-a_n\geq \underline \ell$ and 
$|\lambda_n|\leq \Lambda$. 
\item[(iii)] $a^h(t)$ and $b^h(t)$  are uniformly Lipschitz continuous with respect to $h$ and $t$.
\end{itemize}
\end{proposition}

Before proving this proposition, we show that it implies the existence of a solution (Theorem~\ref{thm:ex}):

\begin{proof}[Proof of Theorem \ref{thm:ex}]

Proposition \ref{prop:properties} implies that $a^h(t)$, $b^h(t)$ and $\ell^h(t)=b^h(t)-a^h(t)$ are  Lipschitz continuous  uniformly with respect to $h$ and that there exists $\underline \ell$ independent of $h$ such that
\begin{equation}\label{eq:llu} 
\underline \ell \leq \ell^h(t) \leq \ell_c \quad \forall t\geq 0.
\end{equation}
In particular there exists a subsequence $h\to0$ and some Lipschitz continuous functions $a(t)$ and $b(t)$ 
such that 
$$ a^h(t) \to a(t), \qquad b^h(t)\to b(t) \mbox{ uniformly with respect to $t\in [a,b]$}$$
for all $a<b\in \RR$.

Next, we note that for any  $t\geq 0$, $u^h(x,t)$ solves 
$$ 
-u_{xx}(\cdot,t) = \lambda(\ell^h(t)) - \kappa\cos\alpha u + (x-b^h(t))\kappa\sin\alpha,  \quad \mbox{ in } (a^h,b^h)
$$
and so  the function $v^h(x,t) = u^h(a^h(t)+x\ell^h(t),t)$ solves 
$$
-v_{xx}(\cdot,t) = {\ell^h(t)}^2\left[ \lambda(\ell^h(t)) - \kappa\cos\alpha v +\ell^h(t) (x-1)\kappa\sin\alpha\right],  \quad \mbox{ in } (0,1)
$$
with
$$ v^h(0,t)=v^h(1,t)=0.$$

In particular, it is readily seen that $x\mapsto v^h(x,t)$ is bounded in $C^2(0,1)$ uniformly with respect to $t$,
and that $v^h$ and $v^h_x$ are Lipschitz  continuous with respect to $t$, uniformly with respect to $h$. Finally, $v^h$ converges  locally uniformly (with respect to $x$ and $t$) to a function $v$ solution of 
$$
-v_{xx}(\cdot,t) = {\ell(t)}^2\left[ \lambda(\ell(t)) - \kappa\cos\alpha v +\ell(t) (x-1)\kappa\sin\alpha\right],  \quad \mbox{ in } (0,1)
\mbox{ with } v(0,t)=v(1,t)=0.$$

Writing  
$$u^h(x,t)=v^h\left(\frac{x-a^h(t)}{\ell^h(t)},t\right),$$ 
we deduce that $u^h$ converges  locally uniformly (with respect to $x$ and $t$) to a function $u$ solution of 
$$
-u_{xx}(\cdot,t) = \lambda(\ell(t)) - \kappa\cos\alpha u + (x-b(t))\kappa\sin\alpha,  \quad \mbox{ in } (a(t),b(t))$$
satisfying
$$u(a(t),t)=u(b(t),t)=0\qquad \mbox{ and }\quad \int u(\cdot,t) dx = V_0.
$$

Furthermore, $u^h_x(a^h(t),t)$ and $u^h_x(b^h(t),t)$ are Lipschitz continuous function in $t$, uniformly with respect to  $h$. 
We can thus pass to the limit in \eqref{eq:md1} and \eqref{eq:md2'}. For instance, we note that \eqref{eq:md1} implies
$$
{b^h}'(t)=\frac{b_{n+1} -b_n }{h } = \frac{1}{2} |u'_n(b_n) | ^2 - \beta(b_n) .
$$
for $t\in (nh,(n+1)h)$, and so the Lipschitz continuity in time implies 
$$
{b^h}'(t)= \frac{1}{2} |u^h_x(b^h(t),t) | ^2 - \beta(b^h(t)) +O(h) \mbox{ a.e. } t\geq 0 
$$
Passing to the limit $t\to 0$ and repeating this argument  with \eqref{eq:md2'} yields \eqref{eq:mmotionc1'}.
\end{proof}

\begin{proof}[Proof of Proposition \ref{prop:properties}] 
For the sake of clarity, we drop the index $n$ when no ambiguity is possible. 

\item[(i)] In view of Poposition \ref{prop:obstacle}, we only need to show that the length of the support $b-a$ cannot be too small:  First, if $u'(a)=0$, then Proposition \ref{prop:obstacle} (iv) implies that $b-a= \ell_c$,  and (\ref{eq:GHV}) gives
$$
\frac{1}{2}|u'(b)|^2= V_0\kappa \sin\alpha.
$$

\vspace{5pt}

Next, if $u'(a)\neq 0$, then (\ref{eq:md1}) and (\ref{eq:md2}) implies that the length of the interval $b-a$ at the next time step will be given by
$$ b-a + \left[\frac{u'(b)^2}{2}+\frac{u'(a)^2}{2}-\beta(a_n)-\beta(b_n)\right] \geq b-a + \left[\frac{u'(b)^2}{2}-\beta(a_n)-\beta(b_n)\right] $$
In particular, if $|u'(b)| \geq 2\sup\sqrt{ \beta}$, then the length of the interval increases and therefore $|u'(b)|$ decreases.
We deduce that if $u'(a)\neq 0$, then
$$|u'(a)|\leq  |u'(b)| \leq \max\{ 2\sup\sqrt\beta, |u_0'(b)|\}$$
and the result follows.

\vspace{10pt}

\item[(ii)]
The proof of (i) above clearly implies that $b-a$ is bounded below. 
Furthermore, integrating the equation satisfied by $u$ over $(a,b)$, we get:
\begin{eqnarray*}
\lambda (b-a) & =  & u'(a) -u'(b) + V_0\kappa\sin\alpha-\frac{1}{2} (b^2-a^2) \kappa \sin\alpha\\
& \leq & 2M + V_0\kappa\sin\alpha.
\end{eqnarray*}
The result follows.

\vspace{10pt}

\item[(iii)] The discrete motion law (\ref{eq:md1}) and (i) imply that there exists a constant $C>0$ such that
$$\left| \frac{b_{n+1}-b_n}{h}\right| \leq C \qquad \mbox{ for all } n\geq 0.$$
Similarly, if $|u_{n+1}'(a_{n+1})|\neq 0$ then  (\ref{eq:md2}) and (i) imply
$$\left| \frac{a_{n+1}-a_n}{h}\right| \leq C \qquad \mbox{ for all } n\geq 0.$$
Finally, if $u'_{n+1}(a_{n+1})= 0$, then  
$b_{n+1}-a_{n+1} = \ell_c$ 
and since $b_{n}-a_{n} \leq \ell_c$ we have
$$ \frac{a_{n+1}-a_n}{h} \leq \frac{b_{n+1}-b_n}{h} \leq C,$$
and  (\ref{eq:md2})  implies
$$  \frac{a_{n+1}-a_n}{h}  \geq \beta(a_n)-\frac{1}{2}M^2 \geq -\frac{1}{2}M^2.$$
The result follows.
\end{proof}

\section{Asymptotic behavior when $\beta$ is constant: Existence of traveling wave solutions}
In this section, we study the long time behavior of the drops when $\beta$ is constant (homogeneous media), and prove Theorem \ref{thm:TW}.

As mentioned earlier, Problem \eqref{eq:P} is much simpler in this case since
the length of the support $\ell(t)=b(t)-a(t)$ solves
\begin{equation}\label{eq:ellev}
 \max\left\{ \ell'(t) -\left[F(\ell(t)) - 2\beta\right] , \ell(t)-\ell_c\right\} = 0
\end{equation}
with $F$ defined by \eqref{eq:Fdef}.
Once the solution of (\ref{eq:ellev}) is found, the solution of  \eqref{eq:P}  is   fully determined by solving 
$$
b'(t) = H(\ell(t)) - \beta .
$$

A traveling waves type solution of (\ref{eq:motionc2})-(\ref{eq:mmotionc1'}) is a solution of the form
$$ u(x,t)=v (x-ct).$$
For such a solution, the endpoints of the support $(a(t),b(t))$ of $u$ satisfy
\begin{equation}\label{eq:c} 
a'(t)=c \mbox{ and } b'(t)=c \mbox{ for all $t$}.
\end{equation}
and the length $b(t)-a(t)$ is constant equal to some $\ell_0$.  Equation (\ref{eq:ellev}) thus reduces to finding $\ell_0$ such that
\begin{equation}\label{eq:elleq}
 \max\left\{-\left[F(\ell_0) - 2\beta\right] , \ell_0-\ell_c\right\} = 0.
\end{equation}
To solve this equation, we recall that $F$ is continuous and monotone decreasing on the interval $(0,\ell_c]$. So we have the following:
\begin{enumerate}
\item Either $F(\ell_c) <2\beta$, in which case, there exists a {\bf unique} $\ell_0<\ell_c$ such that $ F(\ell_0)=2\beta$.
\item Or $F(\ell_c)\geq 2\beta$, in which case we have $\ell_0=\ell_c$.
\end{enumerate}

\medskip

This proves the existence of a unique traveling wave, and we now derive the formula \eqref{eq:vitesse} for the speed $c$:
Let $a$ and $b$ be such that $b-a=\ell_0$ and let $u$ be the corresponding solution of \eqref{eq:uom} with support $(a,b)$.
If $u_x(a)> 0$ , then (\ref{eq:mmotionc1'}) and (\ref{eq:c}) imply
$$ 
\frac{1}{2}|u_x(a)|^2=\beta-c,\quad \frac{1}{2}|u_x(b)|^2=\beta+c
$$
and so (\ref{eq:GHV}) yields
$$ c= \frac{1}{2}(H(\ell_0)-G(\ell_0)) = \frac{1}{2} V_0 \kappa \sin\alpha.$$
This is possible only if $\beta-c=\frac{1}{2}|u_x(a)|^2>0$, that is if 
$ \frac{1}{2} V_0 \kappa \sin\alpha< \beta.$

If  $u_x(a)=0$, then (\ref{eq:c}),
 (\ref{eq:mmotionc1'}) and (\ref{eq:GHV})  imply
$$ c= H(\ell_0)-\beta = V_0\kappa \sin\alpha - \beta.$$
This implies formula \eqref{eq:vitesse}.
\medskip

Finally, the stability of $\ell_0$ is an immediate consequence of the monotonicity of $F$ since it implies that any solution of 
(\ref{eq:ellev}) satisfies
$$ 
\left\{
\begin{array}{l}
\ell'(t)< 0 \mbox{ if } \ell<\ell_0\\
\ell'(t) > 0 \mbox{ if } \ell>\ell_0.
\end{array}
\right.
$$
and so $\ell(t)\to \ell_0$ as $t\to \infty$. Since the profile $u(x,t)$ is completely determined by the length of the support, the convergence of $\ell(t)$ implies the uniform convergence of $u(x,t)$ to the profile of the unique traveling wave.

\section{Asymptotic behavior when $\beta$ is periodic: Existence of pulsating traveling solutions}
In this section, we investigate the long time behavior of the solution when the function $\beta$ is periodic, and prove Theorem \ref{thm:PTF0}. This is of course more delicate than the case where $\beta$ is constant, since we cannot reduce \eqref{eq:P} to a single equation for the length $\ell(t)$ of the support.

First, we observe that for any solution we have (using \eqref{eq:GHV}):
\begin{eqnarray} 
a'(t)+b'(t) & \geq &  \frac{1}{2} |u_x(b(t))|^2 - \frac{1}{2} |u_x(a(t))|^2 +\beta(a(t))-\beta(b(t))\nonumber\\
& \geq  & V_0 \kappa \sin\alpha  +\beta(a(t))-\beta(b(t))\label{eq:a+b}.
\end{eqnarray}
In particular the condition
\begin{equation}\label{eq:nonstat} 
\max\beta -\min\beta <V_0 \kappa \sin\alpha +\delta \quad \hbox{ for some } \delta>0
\end{equation}
guarantees  that $a'(t)+b'(t)>\delta$ for all time.
This implies that no stationary solution can exist under  condition \eqref{eq:nonstat}, and that any drops must slide down the inclined plane.

In order to prove the first part of Theorem \ref{thm:PTF0} (existence and stability of Pulsating traveling solutions), we will prove that \eqref{eq:nonstat} actually implies that $b'(t)>0$ for all (large enough) time, and that we can thus rewrite the equations for $a(t)$ and $b(t)$ using $x=b(t)$ instead of $t$ as a parameter. 
We thus start with the following lemma:
\begin{lemma}\label{lem:b'}
The function $t\mapsto b(t)$ is in $C^{1,1}$ and 
if (\ref{eq:nonstat}) holds then there exists $\eta>0$ such that $b'(t)>\eta$ at least for $t\geq T_0=2\frac{\delta}{\ell(0)}$.
\end{lemma}

\begin{proof}
Let $\eta\in(0,\delta/8)$ be a small number, to be chosen later.
Equation (\ref{eq:a+b}) and condition (\ref{eq:nonstat}) implies
\begin{equation}\label{strict}
a'(t)+b'(t)> \delta, \quad \mbox{ for all } t\geq 0.
\end{equation} 
So as long as $b'(t)\leq \eta$, we have $a'(t) \geq \delta-\eta$ and thus
$$ \ell'(t) = b'(t)-a'(t) \leq -\delta +2\eta \leq -\delta/2.$$
In particular, we cannot have $b'(t)\leq \eta $ in $(0,T_0)$ (with $T_0=2\frac{\delta}{\ell(0)}$), or the length of the droplet would shrink to a point.

Therefore if the lemma is false, then there exists $t_0>t_1$ such that
$b'(t_0)=\eta$ and $b'(t)\geq \eta$ for $t\in (t_1,t_0)$.
Next, we note that $t\mapsto \ell(t)$ and $t\mapsto b(t)$ are Lipschitz continuous and so \eqref{eq:GH} implies that $t\mapsto b'(t)$ is also Lipschitz continuous.
Therefore, there exists $h\leq \frac{\eta}{K}$ (where $K$ is the Lipschitz constant of $b'$) such that
\begin{equation}\label{eq:b'} 
b'(t) <2\eta \quad \mbox{ for } t\in(t_0-h,t_0).
\end{equation}
Proceeding as before (using the fact that $\eta<\delta/8$), we deduce that
$$ \ell'(t) < -\delta/2   \mbox{ for } t\in(t_0-h,t_0)$$
and so
$$ \ell(t_0)<\ell(t_0-h)-\frac{\delta}{2} h$$
(in particular, $\ell(t)<\ell_c$ and so $u_x(a(t),t)>0$  for $t\in(t_0-h,t_0)$). 
Lemma \ref{lem:ndgH} implies
\begin{equation}\label{eq:ndg} 
H(\ell(t_0)) \geq H(\ell(t_0-h)) +c_0 \delta h
\end{equation}
for some $c_0>0$.
Finally, (\ref{eq:b'}) implies
$$ |b(t_0)-b(t_0-h)| \leq  2h\eta $$ 
and so (\ref{eq:GH}) gives 
$$ 
\begin{array}{lcl}
\eta= b'(t_0) = H(\ell_o)  - \beta(b(t_0)) & > & H(\ell(t_0-h))+c_0 \delta h  -\beta(b(t_0-h) - ||\beta'||_\infty   2h\eta \\[8pt]
& > & b'(t_0-h)+c_0 \delta h  -||\beta'||_\infty  2h\eta 
\end{array}
$$
Since we can always choose $h$ such that $b'(t_0-h)>\eta$ (by taking $h\leq t_0-t_1$), we obtain  a contradiction if we choose $\eta$ such that
$$ \eta <  \frac{c_0 \delta}{2||\beta'||_\infty}.$$
\end{proof}

\begin{proof}[Proof of Theorem \ref{thm:PTF0} (i)]
Lemma \ref{lem:b'} implies that $t\mapsto b(t)$ is bijective from $(T_0,\infty)$ onto $(b(T_0),\infty)$. We can thus re-parametrize the motion of the drop using $x=b(t)$ as our new parameter.
We now write the equation satisfied by $y(x)=\ell(b^{-1}(x)$:
$$
\max\{y'(x) b'(b^{-1}(x)) - [F(y(x)) - \beta(a(b^{-1}(x))) -\beta(x)], y(x)-\ell_c\} =0.
$$
Using the fact that $b'(b^{-1}(x))>0$ and  $a(t)=b(t)-\ell(t)$, the equation can be written as
$$
\max\{y'(x) - \frac{1}{ b'(b^{-1}(x))}[F(y(x)) - \beta(x-y(x)) -\beta(x)], y(x)-\ell_c\} =0.
$$
Finally, using the equation for $b'(t)$, we deduce that $y$ solves
\begin{equation}\label{eq:y}
\max\{y'(x) -\mathcal F (y(x),x), y(x)-\ell_c\} =0.
\end{equation}
with
$$ 
\mathcal F(x,y) = \frac{F(y(x)) - \beta(x-y(x)) -\beta(x)}{H(y(x)) -\beta(x)}.
$$
In particular, the function $(x,y)\mapsto \mathcal F(x,y)$ is periodic with respect to $x$.

\vspace{10pt}

We now need the following lemma which will be proved  later on:
\begin{lemma}\label{lem:y}
There exists a $1$-periodic function $z(x)$ such that
$$ \lim_{n\rightarrow \infty} y(n+x)=z(x).$$
\end{lemma}

In order to conclude, we rewrite the equation for $b(t)$ as
\begin{eqnarray*}
 b'(t)& = & H(y(b(t))) - \beta(b(t)) \\
 & =& \mathcal H(b(t))+\vphi(b(t))
 \end{eqnarray*}
with
$\mathcal H(b):=H(z(b))-\beta(b)$ periodic function, and 
$\vphi(b):=H(y(b)) - H(z(b))$ satisfying
$$ |\vphi(b)|\longrightarrow 0 \mbox{ as } b\to\infty.$$

Since $b'(t)\geq \eta$, there exists a sequence $t_n\to\infty$ such that $b(t_n)=n$. Let us consider
$ b_n(t):=b(t_n+t)-n$, which solves the following equation:
$$
b_n'(t) = \mathcal H(b_n(t))+\vphi(b_n(t)+n) ,\qquad b_n(0)=0.$$
We deduce
$$ |b_{n+k}(t)-b_n(t)|\leq e^{Ct} (|\vphi(n)|+|\vphi(n+k)|), $$
and thus $\{b_n\}$ is a Cauchy sequence: it converges to a monotone increasing function $b_\infty(t)$.
Since 
$$ b_n(t_{n+1}-t_n) = b_{n+1}(0)+1,$$
we deduce that $t_{n+1}-t_n$  converges to the unique $t_0$ satisfying $b_\infty(t_0)=1$ and that $b_\infty$ is $t_0$-periodic.

Hence, we can show that $b(t)-b_\infty(t)$ converges to $0$ as $t$ goes to $\infty$, and defining $\ell_\infty(t)=z(b_\infty(t))$, and $a_\infty(t)=b_\infty(t)-\ell_\infty(t)$, we obtain a pulsating solution $(a_\infty(t),b_\infty(t))$. 

Lastly, the uniqueness of the pulsating solution with positive speed (which is guaranteed by Lemma~\ref{lem:b'}) follows from the comparison principle (Proposition~\ref{prop:cp}).

\end{proof}

\begin{proof}[Proof of Lemma \ref{lem:y}]
We recall that $y(x)$ is a bounded solution of (\ref{eq:y}).
For $n\in \NN$, we denote $y_n(x)=y(x+n)$. The periodicity of $\mathcal F$ with respect to $x$ implies that it is also a bounded solution of (\ref{eq:y}) with initial condition $y_n:=y(n)$.

If $y_1=y_0$, then the uniqueness implies that $y(x+1)=y(x)$ for all $x$ and we can take $z(x)=y(x)$.
Otherwise, if, for instance $y_1>y_0$, then the uniqueness implies that 
$ y_1(x)\geq y_0(x)$ for all $x\geq 0$ and so $y_2=y_1(1)\geq y_0(1)-y_1$. We deduce $y_2(x)\geq y_1(x)$, and iterating this argument, we get
$$ y_{n+1}(x)\geq y_n(x) \quad \mbox{ for all $n\in\NN$}.$$
So $\{y_n(x)\}$ is monotone and bounded, and thus there exists $z(x)$ such that $\lim_{n\to\infty} y_n(x)=z(x)$. Since $y_{n}(x+1)=y_{n+1}(x)$, it follows that $z(x+1)=z(x)$, which completes the proof.
\end{proof}

To complete the proof of Theorem \ref{thm:PTF0}, it remains to investigate the case where
\begin{equation}\label{eq:maxmin}
\max\beta - \min\beta \geq V_0\kappa\sin\alpha +\delta
\end{equation}
for some $\delta>0$.
The comparison principle (Proposition \ref{prop:cp}) implies that stationary solutions act as barrier and 
prevents the motion of the drop down the inclined plane.
However, condition \eqref{eq:maxmin} does not automatically imply the existence of a stationary solution.
Nevertheless, we can show  that  stationary solutions do not exist provided the period of $\beta$ is small enough.

\begin{proof}[Proof of Theorem \ref{thm:PTF0} (ii)]
We denote by $L$ the period of $\beta$.

We first fix $\ell_0$ such that $H(\ell_0)=\max \beta$.
We note that we can always find such an $\ell_0$. Indeed, $H(\ell)\to+\infty$ as $\ell\to0$, and using the fact that $G(\ell_c)=0$ together with \eqref{eq:GHV}, we get $H(\ell_c) = V_0 \kappa \sin \alpha$. Condition  \eqref{eq:maxmin} thus implies
$$ H(\ell_c) = V_0 \kappa \sin \alpha <\max \beta -\min\beta <\max \beta.$$
and the intermediate value theorem implies the existence of $\ell_0<\ell_c$ such that $H(\ell_0)=\max \beta$.

We now fix $a$ such that $\beta (a)=\min \beta$, and we find $b\in[a+\ell_0,a+\ell_0+L]$ such that $\beta (b)=\max \beta$ (we can always find such a $b$ thanks to the periodicity of $\beta$).

Consider now the solution of \eqref{eq:uom} with support $(a,b)$.
Since $\ell_0\leq b-a\leq \ell_0+L$, we have 
$$\frac{1}{2}|u_x|^2(b)=H(b-a) \leq H(\ell_0) =\max\beta$$
and so (thanks to the choice of $b$)
\begin{equation}\label{eq:bbb} 
\frac{1}{2}|u_x|^2(b) -\beta (b)\leq 0.
\end{equation}
We also have  (using Lemma \ref{lem:H}, and the fact that $\ell_0<\ell_c$)
\begin{equation}\label{eq:uuuu}
\frac{1}{2}|u_x|^2(a)=G(b-a) \geq G(\ell_0+L) \geq G(\ell_0)-C L,
\end{equation}
as long as $L\leq \frac{1}{2}(\ell_c-\ell_0)$. Finally, (\ref{eq:GHV}) implies
$$ G(\ell_0)=H(\ell_0)-V_0\kappa\sin\alpha  = \max \beta -V_0\kappa\sin\alpha.$$
In particular \eqref{eq:uuuu} yields
$$ \frac{1}{2}|u_x|^2(a) \geq \beta (a)=\min \beta$$
as soon as
$$  \max \beta -\min\beta \geq V_0\kappa\sin\alpha +CL $$
which holds for all $L\leq \delta_0$ in view of \eqref{eq:maxmin}.
Therefore we deduce
\begin{equation}\label{eq:aaa} 
 \beta(a) - \frac{1}{2}|u_x|^2(a) \leq 0
 \end{equation}

The comparison principle together with \eqref{eq:bbb} and \eqref{eq:aaa} implies that if 
$a(0)\leq a$ and $b(0)\leq b$, then the corresponding solution will satisfy
$$ a(t)\leq a \mbox{ and } b(t)\leq b \mbox{ for all }t\geq 0,$$
hence the result.
\end{proof}

\section{Homogenization of the velocity law}\label{sec:hom}
We now prove our last result, Theorem \ref{thm:hom}. Recall that the function $r(q)$ is defined  for all $q\neq 0$ as in \eqref{hom:vel}, as the effective speed of the global (periodic) solution of the ODE
\begin{equation}\label{eq:xq} 
x'(t) = q - \beta(x(t)).
\end{equation} 
(see Lemma \ref{lem:r}, whose proof is provided in appendix \ref{app:r}).

\medskip

Now, let $u^\e(x,t)$ solve the inhomogeneous problem with period $\e$ and with given initial data $u_0$.  We also denote by $a^\eps(t)$ and $b^\eps(t)$ the left and right endpoints of the support of $u^\eps$.

\vspace{20pt}

\noindent {\bf Convergence.} First, we obtain a priori estimate on $u^\eps$ in the same way that we did in the construction of a solution in Section \ref{sec:dt}:

We recall that $a^\eps(t)$, $b^\eps(t)$ and $\ell^\eps(t)=b^\eps(t)-a^\eps(t)$ are Lipschitz continuous (bounded speed of the endpoints) uniformly with respect to $\eps$ and there exists $\underline \ell$ independent of $\e$ such that
\begin{equation}\label{eq:ll1} 
\underline \ell \leq \ell^\eps(t) \leq \ell_c \quad \forall t\geq 0.
\end{equation}

In particular there exists a subsequence $\e\to0$ and some Lipschitz continuous functions $a(t)$ and $b(t)$ 
such that 
$$ a^\eps(t) \to a(t), \qquad b^\eps(t)\to b(t) \mbox{ uniformly with respect to $t\in [a,b]$}$$
for all $a<b\in \RR$.

Next, we note that for any  $t\geq 0$, $u^\eps(x,t)$ solves 
$$ 
-u_{xx}(\cdot,t) = \lambda(\ell^\eps(t)) - \kappa\cos\alpha u + (x-b^\eps(t))\kappa\sin\alpha,  \quad \mbox{ in } (a^\eps,b^\eps)
$$
and so  the function $v^\eps(x,t) = u^\eps(a^\eps(t)+x\ell^\eps(t),t)$ solves 
$$
-v_{xx}(\cdot,t) = {\ell^\eps(t)}^2\left[ \lambda(\ell^\eps(t)) - \kappa\cos\alpha v +\ell^\eps(t) (x-1)\kappa\sin\alpha\right],  \quad \mbox{ in } (0,1)
$$
with
$$ v^\eps(0,t)=v^\eps(1,t)=0.$$

In particular, it is readily seen that $x\mapsto v^\eps(x,t)$ is bounded in $C^2(0,1)$ uniformly with respect to $t$,
and that $v^\eps$ and $v^\eps_x$ are Lipschitz  continuous with respect to $t$, uniformly with respect to $\eps$. Finally, $v^\eps$ converges  locally uniformly (with respect to $x$ and $t$) to a function $v$ solution of 
$$
-v_{xx}(\cdot,t) = {\ell(t)}^2\left[ \lambda(\ell(t)) - \kappa\cos\alpha v +\ell(t) (x-1)\kappa\sin\alpha\right],  \quad \mbox{ in } (0,1)
\mbox{ with } v(0,t)=v(1,t)=0.$$

Writing  
$$u^\eps(x,t)=v^\eps\left(\frac{x-a^\eps(t)}{\ell^\eps(t)},t\right),$$ 
we deduce:
\begin{lemma}
The function $x\mapsto u^\eps(x,t)$ is bounded in $C^2(a^\eps,b^\eps)$ uniformly w.r.t. $t$, and $u^\eps$ and $u^\eps_x$ are Lipschitz  continuous with respect to $t$, uniformly with respect to $\eps$.
Finally, $u^\eps$  converges  locally uniformly (with respect to $x$ and $t$) to a function $v$ solution of 
$$
-u_{xx}(\cdot,t) = \lambda(\ell(t)) - \kappa\cos\alpha u + x\kappa\sin\alpha,  \quad \mbox{ in } (a(t),b(t)) \mbox{ with }u(a(t),t)=u(b(t),t)=0$$
and 
$$
\int u(\cdot,t) dx = V_0.
$$
\end{lemma}

We can now prove Theorem \ref{thm:hom}:

\begin{proof}[Proof of Theorem \ref{thm:hom}] We first consider the equation for $b^\eps(t)$, which reads
$$ {b^\eps}'(t) = H(\ell^\eps(t))-\beta(b^\eps(t)/\eps).$$
Since $t\mapsto \ell^\eps$ is a Lipschitz continuous function, we can also write
$$  {b^\eps}'(t) = q^\eps(t)-\beta(b^\eps(t)/\eps)$$ 
with $q^\eps(t)$ uniformly (w.r.t. $\eps$) Lipschitz function converging (uniformly) to $q(t)=H(\ell(t))$.

For a given $t_0>0$, the function
$$x^\eps(t)=\frac{1}{\eps}b^\eps(t_0+\eps t)$$
solves
$$ {x^\eps}'(t) = q^\eps(t_0+\eps t)-\beta(x^\eps(t)).$$
Since $q^\eps(t_0+\eps t)$ is Lipschitz and $q^\eps(t_0)$ converges to $q(t_0)$, we have that for all $\delta>0$, there exists $\eps_0$ such that
$$ |q^\eps(t_0+\eps t) -q(t_0)| \leq \delta +K\eps t \qquad \forall \eps<\eps_0,\quad \forall t.$$ 
and so
\begin{equation}\label{eq:q}
 |q^\eps(t_0+\eps t) -q(t_0)| \leq 2\delta  \qquad \forall \eps<\eps_0,\quad \forall |t|\leq \frac{\delta}{K\eps}.
 \end{equation}

It follows that
$$ {x^\eps}'(t) \leq  q(t_0) +2\delta -\beta(x^\eps(t)) \qquad \forall \eps<\eps_0,\quad \forall |t|\leq \frac{\delta}{K\eps} .$$
and
$$ {x^\eps}'(t) \geq  q(t_0) -2\delta -\beta(x^\eps(t)) \qquad \forall \eps<\eps_0,\quad \forall |t|\leq \frac{\delta}{K\eps} .$$

Denoting by $\overline x$ (respectively $\underline x$) the solution of \eqref{eq:xq} with $q=  q(t_0) +2\delta$ (respectively $q=  q(t_0) -2\delta$) satisfying $\overline x(0)=\underline x(0)= x^\eps(0)$, we deduce
\begin{equation}\label{eq:tr}
 \underline x(t)\leq x^\eps(t)\leq \overline x(t) \qquad \forall \eps<\eps_0,\quad 0\leq t \leq \frac{\delta}{K\eps} .
 \end{equation}
\begin{equation}\label{eq:tl}
 \overline x(t)\leq x^\eps(t)\leq \underline x(t) \qquad \forall \eps<\eps_0,\quad  -\frac{\delta}{K\eps}\leq t \leq 0.
 \end{equation}

Equation \eqref{eq:tr} implies
$$r(q(t_0) -2\delta) t -1 \leq x^\eps(t)-x^\eps(0)\leq r(q(t_0) +2\delta) t +1\qquad \forall \eps<\eps_0,\quad   0\leq t \leq \frac{\delta}{K\eps} $$
and so
$$r(q(t_0) -2\delta) t -\eps \leq b^\eps(t_0+t)-b^\eps(t_0)\leq r(q(t_0) +2\delta) t +\eps \qquad \forall \eps<\eps_0,\quad  0\leq t \leq \frac{\delta}{K} .$$
Passing to the limit $\eps\to 0$, we deduce
$$r(q(t_0) -2\delta)   \leq \frac{b (t_0+t)-b(t_0)}{t}\leq r(q(t_0) +2\delta)  \qquad  0\leq t \leq \frac{\delta}{K}. $$
and taking the limit $t\to 0^+$ and then $\delta\to0$, we obtain (using the continuity of the function $q\mapsto r(q)$):
$$r(q(t_0))   \leq \liminf_{t\to0_+}\frac{b (t_0+t)-b(t_0)}{t}\leq \limsup_{t\to0_+}\frac{b (t_0+t)-b(t_0)}{t} \leq  r(q(t_0)). $$
Furthermore, using \eqref{eq:tl} instead of \eqref{eq:tr}, we can show that a similar inequality holds for the limit $t\to0^+$.
We thus deduce that $b$ is differentiable at $t_0$ and satisfies
$$b'(t_0)= r(q(t_0))$$
where 
$q(t_0)=H(\ell(t_0)) = \frac{1}{2} |u_x(b(t_0),t_0)|^2$.
\bigskip

The equation for ${a^\eps}(t)$ is handled in a similar fashion. Recall that
$$ -{a^\eps}'(t) \leq G(\ell^\eps(t))-\beta(a^\eps(t)/\eps),$$
with equality if $|u^\eps_x(a^\eps(t))|>0$.

If $\lim_{\eps\to0 }|u^\eps_x(a^\eps(t_0))| >0$, then we have $|u^\eps_x(a^\eps(t))| >0$ for $t$ in a neighborhood of $t_0$ and for small $\eps$, and the argument presented above applies. If $\lim_{\eps\to0} |u^\eps_x(a^\eps(t_0))| =0$, then the above argument applies, but we have $q(t_0)=0$ and so we  only get
$$ 
-a^\eps(t_0+t)+a^\eps(t_0)\leq r(2\delta) t +\eps \qquad \forall \eps<\eps_0,\quad  0\leq t \leq \frac{\delta}{K} .$$
Passing to the limit $\eps\to 0$, we deduce
$$- \frac{a (t_0+t)-a(t_0)}{t}\leq r( 2\delta)  \qquad  0\leq t \leq \frac{\delta}{K}. $$
and taking the limit $t\to 0^+$ and then $\delta\to0$, we obtain (using the continuity of the function $q\mapsto r(q)$):
$$  \limsup_{t\to0_+}- \frac{a (t_0+t)-a(t_0)}{t} \leq  r(0). $$
and a similar  limit for $t\to0^+$.

This prove that $u^\eps$ converges (up to a subsequence) to a solution of \eqref{eq:motionsupp}, \eqref{eq:motionob}, \eqref{eq:velhom}.
\vspace{20pt}

To complete the proof of Theorem \ref{thm:hom}, it remains to prove the uniqueness of the solution of the homogenized problem, which also implies the convergence of the whole sequence $u^\eps$ to $u$. 
For that, 
we will show that a comparison principle similar to that of Proposition~\ref{prop:cp} holds (note that we cannot deduce this comparison principle by passing to the limit $\eps\to0$ in Proposition~\ref{prop:cp} since two solutions of the homogenized problem may be obtained by passing to the limit along different subsequences of $\eps$).

Using the notations of (\ref{not:GH}) we first rewrite \eqref{eq:velhom} as follows
\begin{equation}\label{eq:GHhom}
\left\{
\begin{array}{ll}
\min\{ a'(t)+r(G(\ell(t))),\ell_c-\ell(t)\}=0 \\[5pt]
b'(t) = r(H(\ell(t))).
\end{array}
\right.
\end{equation}
We now consider two solutions of the homogenized problem $u_1$ and $u_2$ with initial data with support $(a_1(0),b_1(0))$ and $(a_2(0),b_2(0))$ satisfying
$$ a_1(0)\leq a_2(0) \, , \quad \mbox{ and } b_1(0)\leq b_2(0).$$
We then define 
$$ f(t) =\max\{a_1(t)-a_2(t),b_1(t)-b_2(t)\}.$$
Since the functions $a_i(t)$ and $b_i(t)$ are in $W^{1,\infty}_{loc}(0,\infty)$, $f(t)$ is in $W^{1,\infty}_{loc}(0,\infty)$ and we have
$$ f'(t)=
\left\{
\begin{array}{ll}
a'_1(t)-a'_2(t) & \mbox{ if } \ell_2 (t)>\ell_1(t) \\[5pt]
b'_1(t)-b'_2(t) & \mbox{ if } \ell_2 (t)\leq \ell_1(t)
\end{array}
\right.
$$
When $\ell_2 (t)>\ell_1(t)$, we have in particular $\ell_1(t)<\ell_c$, and so (\ref{eq:GHhom}) gives
$$ a'_1(t) = -r(G(\ell_1(t))), \quad a'_2(t) \geq -r(G(\ell_2(t))).$$
We deduce
$$
 f'(t)\leq 
\left\{
\begin{array}{ll}
r(G(\ell_2(t)))-r(G(\ell_1(t)))  & \mbox{ if } \ell_2 (t)>\ell_1(t) \\[5pt]
r(H(\ell_1(t)))-r(H(\ell_2(t)))  & \mbox{ if } \ell_2 (t)\leq \ell_1(t)
\end{array}
\right.
$$
and the monotonicity of $G$,  $H$ and $r$ implies 
$$
f'(t)\leq 0 .
 $$
We deduce that if $f(0)\leq 0$, then $f(t)\leq 0$ a.e. $t\geq 0$ which implies the comparison result, and the uniqueness of the solution.
\end{proof}

\appendix

\section{The function $r(q)$: Proof of Lemma \ref{lem:r}}\label{app:r}
In this section, we prove Lemma \ref{lem:r}.
We recall that $\beta$ is a periodic function (with period $1$ for instance) and for a given $q\geq 0$, we consider the following equation:
\begin{equation}\label{eq:ode}
x'(t) = q -\beta(x(t)).
\end{equation}
Since $\beta$ is a Lipschitz function, (\ref{eq:ode}) has a unique solution for any initial data $x(0)=x_0$ and two solutions can never cross.

 \medskip
  
Our first remark is that if $q\in[\min \beta,\max \beta]$, then any solution of (\ref{eq:ode}) will be trapped  in the sense that 
$ x(t)\in[x(0)-1,x(0)+1]$
for all $t$. Indeed, periodicity of $\beta$ implies that there exists $x_1\in[x(0)-1,x(0))$ and $x_2\in(x(0),x(0)+1]$ such that $q-\beta(x_1)\geq 0$ and $q-\beta(x_2)\leq 0$. It is then easy to show that $x(t)\in [x_1,x_2]$ for all $t$.

In that case, the effective speed of any solutions of (\ref{eq:ode}) is zero. We thus have
$$ r(q)=0 \, , \quad q\in [\min \beta,\max \beta].$$
 
 \medskip

For $q>\max \beta$ and $q<\min\beta$, on the other hand, any solution of \eqref{eq:ode} will be strictly monotone, with $\lim_{t\to \pm \infty} x(t)=\pm\infty$.
In particular, we can then show that all solutions are equal, up to a translation in time.
Assuming that $x(0)=0$ (without loss of generality), there exist a unique $t_c$ such that  $x(t_c)=1$. 
Uniqueness for \eqref{eq:ode} implies that $t\mapsto x(t)$ is then periodic with period $t_c$ and 
that its effective speed is given by
$$ r(q)=\frac{1}{t_c}$$

We then have
\begin{proposition}\label{prop:rq}
The function $r:q\mapsto r(q)$ defined above is a non-decreasing function. 
Furthermore,
\item[(i)] $q\mapsto r(q)$ is locally Lipschitz in $(\max \beta,\infty)$ and $[0,\min\beta)$
\item[(ii)]  If $q_0=\max \beta$, then we have
$$ r(q)\leq C(q-q_0)^{1/2} \quad \mbox{ for } q\geq q_0$$
if $x\mapsto \beta$ is $C^2$, and 
$$ r(q)\leq C(-\ln(q-q_0))^{-1} \quad \mbox{ for } q\geq q_0$$
if $x\mapsto \beta$ is only Lipschitz.
In particular $r$ is a continuous function in $[0,\infty)$
\end{proposition}
\begin{proof}[Proof of Proposition \ref{prop:rq}]
To prove the first part, we consider $q_2>q_1>\max \beta$ and we set $\eta=q_2-q_1$. Let $x(t)$ and $y(t)$ be the solution of \eqref{eq:ode} with $q=q_1$ and $q=q_2$ respectively, and $x(0)=y(0)=0$.
Since $x'(t)\geq q_1-\max\beta=\delta>0$, we can define
$$ h(s)=y\circ x^{-1}(s)-s,$$
solution of 
\begin{eqnarray*}
 h'(s) & = & \frac{q_1-\beta(x+h(x))+\eta}{q_1-\beta(x)}-1\\
 & = & \frac{\beta(x)-\beta(x+h(x))+\eta}{q_1-\beta(x)}\\
 & \leq & \frac{K}{\delta}(h(x)+\eta)
 \end{eqnarray*}
Since $h(0)=y(x^{-1}(0))-0=0$,  the Gronwall's lemma implies
\begin{equation}\label{eq:h1} 
h(1) \leq C \eta.
\end{equation}
If we denote $t_c$ such that $x(t_c)=1$ (so that $r(q_1)=1/t_c$), then (\ref{eq:h1}) implies
$$ y(t_c)\leq 1+C\eta.$$
Finally, since $y'(t)\geq \delta$, this implies $t_2$, such that $y(t_2)=1$ satisfies $t_2  \geq t_c -C\eta$ and so
$$ r(q_2)\leq r(q_1)+C\eta$$
for some $C$ depending on $q_1$ and $r(q_1)$.

\medskip

The constant $C$ degenerates when $r(q)$ goes to zero, so we need different argument to prove (ii):
Let $\eta = q-q_0$. Assume (with loss of generality) that $\max \beta=\beta(0)$ (and so $\beta'(0)=0$). If $\beta$ is $C^2$, then
$$ \beta(x) \geq \beta(0)  -Cx^2 = q_0-Cx^2, \quad \mbox{ for all } x$$
and so
$$ q -\beta(x) \leq \eta +Cx^2.$$

Let now $x(t)$ be the corresponding solution of \eqref{eq:ode}. Up to a translation in time, we can always assume that $x(0)=0$. We then have
$$ x'(t)\leq \eta+Cx(t)^2$$
We deduce
$$ x(t)\leq \frac{C\eta t}{1-C\sqrt\eta t} \quad \mbox{ for } t\leq (C\sqrt \eta)^{-1}.$$ 
In particular, $t_c$, defined by $x(t_c)=1$ satisfies
$$ \frac{1}{t_c}\leq C(\eta+\sqrt\eta)$$
hence the first result.

When $\beta$ is only Lipschitz, a similar argument yields
$$ x'(t)\leq \eta+Cx(t)$$
 which gives the second result.
\end{proof}

\end{document}